\documentclass[10pt]{scrartcl}

\usepackage{amsmath, amssymb, amsthm, marginnote, ltxcmds, adjustbox, stmaryrd, graphicx, bbold, stackengine}

\usepackage{mathtools, stackengine, changepage, ragged2e}

\usepackage{todonotes}

\definecolor{cite}{HTML}{11871E}
\definecolor{url}{HTML}{698996}
\definecolor{link}{HTML}{912F1B}

\usepackage[pdfencoding=unicode, colorlinks=true, linkcolor=link, citecolor=cite, urlcolor=url, linktocpage]{hyperref}

\usepackage[backend=biber, style=alphabetic, maxnames=10, maxalphanames=3, minalphanames=3]{biblatex}
\usepackage{tikz}
\usetikzlibrary{cd}
\usetikzlibrary{arrows, arrows.meta, positioning, calc,decorations.pathmorphing}
\tikzcdset{arrow style=tikz, diagrams={>={Straight Barb[scale=0.8]}}}

\tikzstyle{arrow} = [-{Straight Barb[scale=0.8]}, line width=0.2mm]

\usepackage{enumitem}
{\end{enumerate}%
}

\usepackage[
	letterpaper,
	twoside=false,
	textheight=22cm,
	textwidth=14.4cm,
	marginparsep=0.75cm,
	marginparwidth=2.5cm,
	heightrounded,
	centering
]{geometry}

\usepackage{stmaryrd}

\usepackage[capitalise]{cleveref}

\Crefname{proposition}{Proposition}{Propositions}
\Crefname{lemma}{Lemma}{Lemmas}
\Crefname{corollary}{Corollary}{Corollaries}
\Crefname{theorem}{Theorem}{Theorems}
\Crefname{alphThm}{Theorem}{Theorems}
\Crefname{alphCor}{Corollary}{Corollaries}

\Crefname{definition}{Definition}{Definitions}
\Crefname{notation}{Notation}{Notations}
\Crefname{construction}{Construction}{Constructions}
\Crefname{remark}{Remark}{Remarks}
\Crefname{observation}{Observation}{Observations}
\Crefname{trick}{Trick}{Tricks}
\Crefname{warning}{Warning}{Warnings}
\Crefname{conj}{Conjecture}{Conjectures}
\Crefname{assump}{Assumption}{Assumptions}
\Crefname{recollection}{Recollection}{Recollections}
\Crefname{terminology}{Terminology}{Terminologies}
\Crefname{convention}{Convention}{Conventions}

\Crefname{question}{Question}{Questions}
\Crefname{example}{Example}{Examples}

\Crefname{figure}{Figure}{Figures}

\crefformat{equation}{(#2#1#3)}
\crefformat{section}{\S#2#1#3}
\crefmultiformat{section}{\S\S#2#1#3}{and~#2#1#3}{, #2#1#3}{, and~#2#1#3}

\newtheorem{theorem}[subsection]{Theorem}
\newtheorem{proposition}[subsection]{Proposition}
\newtheorem{lemma}[subsection]{Lemma}
\newtheorem*{lem*}{Lemma}

\newtheorem{alphThm}{Theorem}

\newcommand{\neutralize}[1]{\expandafter\let\csname c@#1\endcsname\count@}
\makeatother

\newtheorem{alphCor}{Corollary}



\theoremstyle{definition}
\newtheorem{definition}[subsection]{Definition}
\newtheorem{construction}[subsection]{Construction}
\newtheorem{notation}[subsection]{Notation}

\newtheorem{recollection}[subsection]{Recollections}

\newtheorem{setting}[subsection]{Setting}
\newtheorem{convention}[subsection]{Convention}
\newtheorem{condition}[subsection]{Condition}

\theoremstyle{remark}
\newtheorem{remark}[subsection]{Remark}
\newtheorem{observation}[subsection]{Observation}
\newtheorem{example}[subsection]{Example}

\newcommand{\inerts}{^{\mathrm{int}}}
\newcommand{\elementary}{^{\mathrm{el}}}
\newcommand{\map}{\mathrm{Map}}
\newcommand{\calg}{\mathrm{CAlg}}

\newcommand{\spectra}{\mathrm{Sp}}
\newcommand{\cmonoid}{\mathrm{CMon}}
\newcommand{\monoid}{\mathrm{Mon}}
\newcommand{\cgroup}{\mathrm{CGrp}}
\newcommand{\spc}{\mathcal{S}}
\newcommand{\res}{\mathrm{Res}}
\newcommand{\sphere}{\mathbb{S}}
\newcommand{\func}{\mathrm{Fun}}
\newcommand{\cat}{\mathrm{Cat}}
\newcommand{\sC}{\mathcal{C}}
\newcommand{\D}{\mathcal{D}}
\newcommand{\mackey}{\mathrm{Mack}}
\newcommand{\orbit}{\mathcal{O}}
\newcommand{\loops}{\Omega^{\infty}}
\newcommand{\cofree}{\mathrm{Cofree}}
\newcommand{\op}{^{\mathrm{op}}}
\newcommand{\finite}{\mathrm{Fin}}
\newcommand{\constant}{\mathrm{const}}
\newcommand{\eval}{\mathrm{ev}}
\newcommand{\forget}{\mathrm{fgt}}
\newcommand{\norm}{\mathrm{N}}

\newcommand{\id}{\mathrm{id}}
\newcommand{\ttimes}{^{\underline{\times}}}
\newcommand{\proper}{\mathcal{P}}

\newcommand{\arrdisp}{0.33ex}
\newcommand{\arrdisplacementsp}{0.72ex}

\newcommand{\ardis}{\ar@<\arrdisp>}
\newcommand{\ardissp}{\ar@<\arrdisplacementsp>}

\usepackage{contour}
\usepackage{ulem}

\contourlength{0.5pt}

\newcommand{\myuline}[1]{%
  \uline{\phantom{#1}}%
  \llap{\contour{white}{#1}}%
}

\makeatletter
\newcommand*{\saved@myuline}{}
\let\saved@myuline\myuline

\newcommand*{\mathuline}{%
  \mathpalette{\math@myuline\saved@myuline}%
}
\newcommand*{\math@myuline}[3]{%
  \mbox{#1{$#2#3\m@th$}}%
}

\renewcommand*{\myuline}{%
  \relax  
  \ifmmode
    \expandafter\mathuline
  \else
    \expandafter\saved@myuline
  \fi
}
\makeatother

\title{\LARGE An equivariant generalisation of McDuff--Segal's group--completion theorem}
\date{November 22, 2023\thanks{Date of publication in \href{https://academic.oup.com/imrn/advance-article/doi/10.1093/imrn/rnad278/7442065?login=false&utm_source=advanceaccess&utm_campaign=imrn&utm_medium=email}{IMRN}.}}
\author{\textsc{Kaif Hilman}\thanks{Max Planck Institute for Mathematics, Bonn\\kaif@mpim-bonn.mpg.de}}

\begin{document}
\maketitle

\begin{abstract}
In this short note, we prove a $G$--equivariant generalisation of McDuff--Segal's group--completion theorem for finite groups $G$. A new complication regarding genuine equivariant localisations arises and we resolve this by isolating a simple condition on the homotopy groups of $\mathbb{E}_{\infty}$--rings in $G$--spectra. We check that this condition is satisfied when our inputs are a suitable variant of $\mathbb{E}_{\infty}$--monoids in $G$--spaces via the existence of multiplicative norm structures, thus giving a localisation formula for their associated $G$--spherical group rings. 
\end{abstract}

\tableofcontents

\section{Introduction}
Group--completion is an important procedure in higher algebra for at least two reasons: (1) it is the main ingredient in constructing the K--theory of symmetric monoidal ($\infty-$)categories; (2) it allows one to port spectral methods to study questions regarding moduli spaces. The homotopy types of these group--completions are however mysterious in general, and the group--completion theorem of McDuff--Segal \cite{mcDuffSegal,randalWilliams} is a classical tool giving a homological formula for these objects. By now, the theorem has become a standard component, for example, in the active area burgeoning in the wake of the Madsen--Weiss theorem (cf. \cite[\S7.4]{galatiusRWActa} and \cite[\S7]{galatiusRWAnnals}) relating the homology of diffeomorphism groups to something amenable to stable homotopy theoretic methods. Very roughly speaking, the strategy is first to show that the group--completion of a geometrically defined cobordism category associated to the diffeomorphism groups is equivalent to a particular Thom spectrum. One then combines this identification with the group--completion theorem to compute, up to stabilisation, the homology of said diffeomorphism groups in terms of the homology of the Thom spectrum. 

In this article, we investigate a $G$--equivariant generalisation of this classical result for finite groups $G$. This is not as contrived a question as it may first seem since one of the main steps for an equivariant generalisation of the ``Madsen--Weiss program'' above has already been explored in \cite[Thm. 1.1]{galatiusSzucs} where they identified the group--completion of a certain equivariant cobordism category with an equivariant Madsen--Tillmann spectrum. Our hope is that the result we present here could provide one of the standard pieces in a future equivariant story and serve as a useful tool for making Bredon homological analyses of equivariant group--completions. 

\begin{convention}
    In this paper, by a category we will always mean an $\infty$--category in the sense of \cite{lurieHA}. When emphasising that something is a category in the classical sense, we will term it as a \textit{1--category}.
\end{convention}

\begin{notation}\label{nota:introSecondNotation}
    We will briefly introduce some notions so as to be able to state the main theorem. More details on all these can be found in \cref{section:Foundations}. We write $\orbit_G$ for the orbit category of the finite group $G$ and $\spc_G\coloneqq \func(\orbit_G\op,\spc)$ for the category of \textit{genuine $G$--spaces}, and write $\cmonoid(\spc_G)\simeq \func(\orbit_G\op,\cmonoid(\spc))$ for the category of $\mathbb{E}_{\infty}$--monoid objects therein. An object $M\in \cmonoid(\spc_G)$ consists of $\mathbb{E}_{\infty}$--monoid spaces $M^H$ for every subgroup $H\leq G$ and the restriction map $M^H\rightarrow M^K$ associated to a subconjugation $K\leq H$ is a map of $\mathbb{E}_{\infty}$--monoids. 
    
    There is a variant with more equivariant structure, namely the category $\cmonoid_G(\underline{\spc}_G)$ of \textit{$G$--$\mathbb{E}_{\infty}$--monoids in genuine $G$--spaces}. An object $M\in\cmonoid_G(\underline{\spc}_G)$ consists of the data above together with ``equivariant addition'' maps $\oplus_{H/K} \colon M^K\rightarrow M^H$ for every $K\leq H$ satisfying double--coset formulas and higher coherences. This turns out, as we shall recall at the end of \cref{cons:equivariantGroupCompletion}, to be equivalent to the category $\mackey_G(\spc)\coloneqq \func^{\times}(A^{\mathrm{eff}}(G),\spc)$ of $G$--Mackey functors valued in spaces defined as product--preserving presheaves on Barwick's effective Burnside category (cf. \cite[Rmk. 2.3]{CMMN2}). There is a forgetful functor $\forget \colon \cmonoid_G(\underline{\spc}_G) \rightarrow \cmonoid(\spc_G)$ forgetting the equivariant addition maps. 

\end{notation}

While we reserve the more general - but notationally heavier - statement of the main result \cref{mainThmA} in the body of the paper, we can however extract the following simple consequence on Bredon homology here (whose proof is given at the end of \cref{section:mainTheorem} after the proof of \cref{mainThmA}): 

\begin{theorem}\label{thm:Bredon}
    Let $M\in\cmonoid_G(\underline{\spc}_G)$ and $\underline{N}$ a $G$--Mackey functor valued in abelian groups. For any $K\leq G$, we have a natural isomorphism of $RO(K)$--graded Bredon homology with $\underline{N}$ coefficients
    \[H_{\star}^K(\Omega BM; \underline{N}) \cong H_{\star}^K(M; \underline{N})[(\pi_0M^K)^{-1}]\]
\end{theorem}

The reader might now justifiably wonder how common $G$--$\mathbb{E}_{\infty}$--monoid $G$--spaces actually are. To address this point somewhat, we will recall a standard mechanism to produce plenty of interesting examples in \cref{example:GMonoidSpaces}.

We now look ahead slightly to say a few words about what is actually proved in \cref{mainThmA} and the methods involved. The general formulation is in terms of higher algebraic localisations of spherical monoid rings (following that of Nikolaus \cite{nikolaus}) and the result will be in two parts: in part (i), we show using a direct adaptation of the proof in \cite[Thm. 1]{nikolaus} that for $M\in\cmonoid(\spc_G)$, the $G$--suspension spectrum of its group completion is computed as an abstract localisation satisfying a universal property. The crux of the matter here is that, unlike the nonequivariant case where one can prove that the abstract localisation can always be identified with a \textit{telescopic} localisation as appears in \cref{thm:Bredon} (cf. for example \cite[App. A]{nikolaus} for the proof of this in the general case of $\mathbb{E}_1$--rings satisfying the Ore condition), this is \textit{not} so in the equivariant setting. However, we do show in part (ii) of \cref{mainThmA} that when $M$ has the additional structure of a $G$--$\mathbb{E}_{\infty}$--monoid $G$--space, the associated $G$--spherical monoid ring attains the structure of the multiplicative norms (in the sense of \cite{greenleesMayMU, HHR}) which in turn ensures that the abstract and telescopic localisations agree. In fact, we will isolate a simple condition on the equivariant homotopy groups of $M$ we call \textit{torsion--extension} (cf. \cref{condition:torsionExtensions}) which ensures that the abstract and telescopic localisations agree even in the absence of the norms. This might be usable and useful in specific cases of $M$.

As far as we know, the theorem cannot be directly deduced from the classical group--completion theorem because the $G$--suspension spectrum of a $G$--space is \textit{not} given simply by taking the suspension spectrum on each genuine fixed points of the $G$--space. The first part of the theorem will require only standard $\infty$--category theory (essentially the same proof as \cite[Thm. 1]{nikolaus} as pointed out above), whereas in the more highly structured second part of \cref{mainThmA} we will need the language of \textit{$G$--categories} introduced in \cite{parametrisedIntroduction} in order to discuss $G$--$\mathbb{E}_{\infty}$  structures succintly. To our untrained eyes, the relevance of the multiplicative norms came as a bit of a surprise, but in hindsight, this result is likely known or at least expected among experts. While we were not able to find this result in the literature, we very much welcome a reference to where this result might have previously appeared and give the appropriate credits.

Lastly, a few words on organisation: we will briefly record some foundational materials in \cref{section:Foundations} to orient the reader who might not be familiar with the formalism of $G$--categories; in \cref{section:mainTheorem}, we give a proof of the main \cref{mainThmA}; and in \cref{section:finalRemarks} we will end the main body of the article with some remarks on how norms and localisations managed to interplay well in our situation and how this result fits in with the nonequivariant group--completion theorem. Along the way, we will explain how geometric fixed points turn the mysterious localisation $L_{\underline{S}^{-1}}R$ into something familiar. We also record a generic situation where this theorem might be useful and give a rich source of examples of $G$--$\mathbb{E}_{\infty}$--monoid $G$--spaces. Finally, in \cref{appendix}, we will prove a technical folklore result which we use in proving the main theorem, namely that $G$--$\mathbb{E}_{\infty}$ algebras in $G$--cartesian symmetric monoidal $G$--categories are the same as $G$--$\mathbb{E}_{\infty}$--monoids in said $G$--category. We have unfortunately not been able to find this in the literature and hope that this appendix will serve to fill in this gap. 

\subsection*{Acknowledgements}

We thank J.D. Quigley and Eva Belmont for posing the question as to what an ``equivariant group--completion theorem'' should be which  directly led us to writing this note. We also thank Maxime Ramzi for going through a draft and for helpful suggestions. Finally, we are also grateful to the anonymous referee whose patient comments led to a minor correction, many expositional improvements, as well as the writing of the appendix.

\section{Some preliminaries}\label{section:Foundations}
Let $G$ be a finite group.
\begin{notation}
    Let $\orbit_G$ be the orbit category of the finite group $G$: this is a 1--category whose objects are transitive $G$--sets and morphisms are $G$--equivariant maps. We write $\spc_G$ for the category of genuine $G$--spaces, which is defined to be $\spc_G\coloneqq \func(\orbit_G\op,\spc)$ where $\spc$ is the category of spaces, and we write $\spectra_G$ for the category of genuine $G$--spectra, a model of which is given by $G$--Mackey functors valued in spectra (cf. \cite{barwick1,barwick2}). We will also denote by
    $\sphere_G[-]$ for the functor ${\Sigma}^{\infty}_{+G}\colon \spc_G\rightarrow \spectra_G$ given by taking the $G$--suspension spectrum.
\end{notation}

\begin{notation}\label{nota:prelimSecondNotation}
    For a category $\sC$ admitting finite products, we write $\cmonoid(\sC)$ for the category of $\mathbb{E}_{\infty}$--monoids in $\sC$; for a symmetric monoidal category $\D^{\otimes}$, we write $\calg(\D^{\otimes})$ for the $\mathbb{E}_{\infty}$--algebra objects in $\D$ under the endowed tensor product structure. Writing $\sC^{\times}$ for the cartesian symmetric monoidal structure, we then have by \cite[Prop. 2.4.2.5]{lurieHA} that $\calg(\sC^{\times})\simeq \cmonoid(\sC)$. Note that this means $\calg(\spectra_G^{\otimes})$ denotes $\mathbb{E}_{\infty}$--rings in genuine $G$--spectra \textit{without} the multiplicative norms.
\end{notation}

We begin with the following observation, which requires no theory of $G$--categories:

\begin{observation}\label{obs:nonequivariantAdjunction}
    It is a standard fact that the left adjoint in the adjunction
    \begin{center}
        \begin{tikzcd}
            \spc_G \rar[shift left = 1, "\sphere_G{[}-{]}"] & \spectra_G\lar[shift left = 1, "\loops_G"] 
        \end{tikzcd}
    \end{center}
    refines to a symmetric monoidal functor with the cartesian symmetric monoidal structure on $\spc_G$ and the tensor product of $G$--spectra on $\spectra_G$ (cf. for example \cite[\S 5.2]{mathewNaumannNoelI} for the case of pointed $G$--spaces, which can then be precomposed with the symmetric monoidal functor of adding a disjoint basepoint $(-)_+\colon \spc_G\rightarrow\spc_{G*}$). Thus by \cite[Cor. 7.3.2.7]{lurieHA} the right adjoint $\loops_G$ automatically refines to a lax symmetric monoidal functor and hence, by \cite[Rmk. 7.3.2.13]{lurieHA}, applying the functor $\calg(-)$ yields an adjunction
    \begin{equation}\label{mainAdjunction}
        \begin{tikzcd}
            \cmonoid(\spc_G)\simeq \calg(\spc_G^{\times}) \rar[shift left = 1, "\sphere_G{[}-{]}"] & \calg(\spectra_G^{\otimes})\lar[shift left = 1, "\loops_G"] 
        \end{tikzcd}
    \end{equation}
\end{observation}

Now, to set the stage for our discussions about the multiplicative norms, we collect here some basics on $G$--categories. The reader uninsterested in this refinement can skip right away to the proof of the first part of \cref{mainThmA} in the next section.

\begin{setting}[The theory of $G$--categories]\label{setting:GCategories}
    In keeping with the tradition of papers about group--completions, we aim to keep this article as compact as possible. As such, we have chosen to travel light in this document and we refrain from giving a self--contained exposition of the required theory on $G$--categories. For the original sources of these materials, we refer the reader to \cite{parametrisedIntroduction,nardinExposeIV, shahThesis}, and a one--stop survey of $G$--categories can be found for example in \cite[ Chap. 1]{kaifThesis}. In short, a $G$--category (resp. a $G$--functor) is an object (resp. morphism) in $\cat_{\infty,G}\coloneqq \func(\orbit_G\op,\cat_{\infty})$ and we will use the underline notation $\underline{\D}$ to denote a $G$--category and $\D_{H}$ for its value at $G/H\in\orbit_G\op$. For subgroups $K\leq H$ of $G$, we should think of the datum $\D_H\rightarrow \D_K$ packaged in the $G$--category $\underline{\D}$ as a ``restriction'' functor $\res^H_K$. In particular, by definition of morphisms in functor categories, a $G$--functor is always compatible with these ``restriction'' maps. Important examples of $G$--categories include genuine $G$--spaces $\{\underline{\spc}_G \colon G/H \mapsto \spc_H\}$ and genuine $G$--spectra $\{\myuline{\spectra}_G\colon G/H\mapsto \spectra_H\}$. Additionally, the functor $\orbit_H\simeq (\orbit_G)_{/(G/H)}\rightarrow (\orbit_G)_{/(G/G)}\simeq \orbit_G$ induces a functor $\cat_{\infty,G}\rightarrow \cat_{\infty,H}$ via restriction which we denote by $\res^G_H$. Using Lurie's notion of \textit{relative adjunctions} \cite[\S 7.3.2]{lurieHA}, one can define the notion of \textit{$G$--adjunctions} (cf. \cite[Def. 8.3]{shahThesis}): this roughly means a pair of $G$--functors $L\colon \underline{\sC}\rightleftharpoons\underline{\D} : R$ together with the data of adjunctions when evaluated at each $G/H\in\orbit_G\op$.
    
    Central to this theory is the notion of $G$--(co)limits, and among these the special cases of \textit{indexed (co)products} find a distinguished place. In this article, we will only need these special cases, and so we briefly explain them now. Intuitively, they should be thought of as taking (co)products with respect to finite $G$--sets so that for example, for $\underline{\sC}\in\cat_{\infty,G}$, $H\leq G$ and a $H$--equivariant object $X \in \sC_H$, $\prod_{G/H}X$ is now a $G$--equivariant object. We refer the reader to \cite[\S5]{shahThesis} for more details on this. When $\underline{\sC}$ is pointed (which just means that $\sC_K$ is pointed for every $K\leq G$ and all the restriction maps preserve the zero objects), one can construct a canonical comparison map $\coprod_{G/H}\rightarrow \prod_{G/H}$ (cf. \cite[Cons. 5.2]{nardinExposeIV}). If this map is an equivalence, then we say that $\underline{\sC}$ is \textit{$G$--semiadditive.} As in the nonequivariant case, for a $G$--category $\underline{\sC}$ with finite indexed products, we may construct (see for instance \cite[Def. 5.9]{nardinExposeIV}) the $G$--semiadditive $G$--category $\underline{\cmonoid}_G(\underline{\sC})$ of \textit{$G$--commutative monoids in $\underline{\sC}$} whose objects should roughly be thought of as objects $M\in\underline{\sC}$ equipped with ``equivariant addition maps'' $\prod_{G/H}\res^G_HM\rightarrow M$ for all $H\leq G$ on top of the usual addition maps $M\times M\rightarrow M$. Observe that this version of the equivariant addition maps recovers the one mentioned in \cref{nota:introSecondNotation} upon applying $(-)^G$ since $M^H\simeq (\prod_{G/H}\res^G_HM)^G \rightarrow M^G$.

Now, denote by $\underline{\finite}_{*}$ for the $G$--category of finite pointed $G$--sets. That is, it is the $G$--category  $\{G/H\mapsto \finite_{*H} \coloneqq \func(BH, \finite_*)\}$ where $BH$ is the groupoid with one object and morphism set given by the group $H$. Nardin used this to give a definition of $G$--symmetric monoidal categories in \cite[$\S3$]{nardinThesis} much like the nonequivariant situation from \cite{lurieHA}.  See also \cite[\S2]{nardinShah} for a comprehensive, more recent treatment and \cite[\S 5.1]{shahQuigley} for a summary of these matters. Suffice to say, in this setting, a $G$--symmetric monoidal category is a $G$--category $\underline{\D}^{\underline{\otimes}}$ equipped with a map to $\underline{\finite}_{*}$ satisfying appropriate cocartesianness and $G$--operadic conditions, and \textit{$G$--$\mathbb{E}_{\infty}$--ring objects} $\calg_G(\underline{\D}^{\underline{\otimes}})\coloneqq \func_{G/\underline{\finite}_*}\inerts(\underline{\finite}_{*}, \underline{\D}^{\underline{\otimes}})$ are then $G$--inert sections to this map (see also \cref{recollection:algebraCategories} for slightly more details to this). An object $R\in \calg_G(\myuline{\D}^{\underline{\otimes}})$ should be thought of as an object $R\in\calg(\D_G^{\otimes})$ equipped with $\mathbb{E}_{\infty}$--algebra maps $\bigotimes^G_H\res^G_HR\rightarrow R$ encoding ``equivariant multiplication''. In this notation, $\calg_G(\myuline{\spectra}_G^{\underline{\otimes}})$ will therefore mean those $\mathbb{E}_{\infty}$--rings in genuine $G$--spectra equipped with multiplicative norms, to be contrasted with objects in $\calg(\spectra_G^{\otimes})$ which do not have norms. Moreover, following \cite{HHR}, we use the notation $\norm^G_H$ instead of $\bigotimes^G_H$ in the special case of $\spectra_G$.

Analogously to \cref{nota:prelimSecondNotation}, denoting by $\underline{\sC}\ttimes$ the $G$--cartesian symmetric monoidal structure on a $G$--category $\underline{\sC}$ which admits finite indexed products, we also have that $\calg_G(\underline{\sC}^{\underline{\times}})\simeq\cmonoid_G(\underline{\sC})$. This is essentially because for $M\in \calg_G(\underline{\sC}^{\underline{\times}})$, the structure $\prod_{G/H}\res^G_HM= \bigotimes^G_H\res^G_HM\rightarrow M$ supplies precisely the ``equivariant addition'' structure to be an object in $\cmonoid_G(\underline{\sC})$. While this is a folklore result, we have not been able to find a proof of this in the literature and so we have indicated a proof in the appendix, see \cref{prop:CMon=CAlg}, where we also give more precise explanations and references for some of the matters discussed above.  
\end{setting}

\begin{lemma}\label{lem:basicAdjunctionCMonNalg}
    The $G$--adjunction $\sphere_G[-] \colon \underline{\spc}_G\rightleftharpoons \myuline{\spectra}_G : {\Omega}^{\infty}_G$ induces an adjunction $\sphere_G[-] \colon {\cmonoid}_G(\underline{\spc}_G) \rightleftarrows {\calg}_G(\myuline{\spectra}_G^{\underline{\otimes}}) : {\Omega}^{\infty}_G$.
\end{lemma}
\begin{proof}
    We know by \cite[$\S3$]{nardinThesis} that the map $\sphere_G[-]$ refines to a $G$--symmetric monoidal functor $\sphere_G[-] \colon \underline{\spc}_G^{\underline{\times}}\longrightarrow \myuline{\spectra}_G^{\underline{\otimes}}$.     This means that ${\Omega}^{\infty}_G$ canonically refines to a $G$--lax symmetric monoidal functor.  Hence using that ${\calg}_G(\underline{\spc}_G^{\underline{\times}})\simeq {\cmonoid}_G(\underline{\spc}_G)$ from \cref{prop:CMon=CAlg} and \cite[Lem. 1.3.11]{kaifThesis} that applying $\calg_G$ yields another adjunction analogously to \cite[Rmk. 7.3.2.13]{lurieHA}, we get the desired adjunction.
\end{proof}

\begin{construction}[Equivariant group--completions]\label{cons:equivariantGroupCompletion}\label{lem:groupCompletion}
    As explained for instance in \cite[\S 1]{GGN}, it makes sense to speak of  objects in arbitrary semiadditive  categories (or \textit{preadditive}, as it was termed in that paper) having the property of being \textit{group--complete} by requiring a certain canonically constructed shear map to be an equivalence. In the case of the semiadditive category $\cmonoid(\spc)$, we will write $\cgroup(\spc)$ for the full subcategory of group--complete objects. One characterisation for some $M\in\cmonoid(\spc)$ to lie in $\cgroup(\spc)$ is that the abelian monoid $\pi_0M$ has the property of being a group. 
    
    If we write $B$ for the suspension in the category $\cmonoid(\spc)$ (which is very different from the suspension on the underlying space!), then we know that we have the adjunction $\Omega B\colon \cmonoid(\spc)\rightleftharpoons \cgroup(\spc) : \mathrm{incl}$  so that $\Omega B$ implements the group--completion functor on $\cmonoid(\spc)$. In fact, as slickly explained in \cite[Prop. 3.3.5]{nineAuthorsII}, $\Omega B$ \textit{always} implements group--completions in \textit{any} semiadditive category with pullbacks and pushouts. We can then cofreely make this into a $G$--adjunction of $G$--categories
    \begin{equation}\label{eqn:cofreeCMonCGrpAdjunction}
        \begin{tikzcd}
        \underline{\cofree}_G\big(\cmonoid(\spc)\big) \rar[shift left = 1, "\Omega B"] & \underline{\cofree}_G\big(\cgroup(\spc)\big) \lar[shift left = 1, hook]
        \end{tikzcd}
    \end{equation}
    
    \noindent Here for a nonequivariant category $\sC$, $\underline{\cofree}_G(\sC)\in\cat_{\infty,G}$ is the $G$--category given by $\{G/H\mapsto \func(\orbit_H\op,\sC)\}_{H\leq G}$ (cf. \cite[Def. 2.7]{nardinExposeIV} for instance, where the notation was just an underline instead of the word ``cofree''). In this notation, the $G$--category of $G$--spaces is therefore given by $\underline{\spc}_G = \underline{\cofree}_G(\spc)$. Both $G$--categories in the adjunction are $G$--semiadditive, and so in particular $\Omega B$ preserves $G$--biproducts. Now if $\sC$ admits finite products, then writing $\mackey_G(-)\coloneqq \func^{\times}(A^{\mathrm{eff}}(G),-)$ for $G$--Mackey functors, we have
    \[\cmonoid_G(\underline{\cofree}_G(\sC))\simeq \mackey_G(\sC)\simeq \mackey_G(\cmonoid(\sC))\]
    where the first equivalence is by \cite[Thm. 6.5]{nardinExposeIV} and the second is well--known and can be deduced for example from \cite[Thm. II.19]{hebestreitWagner}, using that $A^{\mathrm{eff}}(G)$ is semiadditive by \cite[Prop. 4.3, Ex. B.3]{barwick1}. Using this, we can then apply ${\cmonoid}_G(-)$ to the adjunction \cref{eqn:cofreeCMonCGrpAdjunction} to get an adjunction
    \begin{center}
        \begin{tikzcd}
            {\cmonoid}_G(\underline{\spc}_G) \rar[shift left = 1, "\Omega B"] & {\cgroup}_G(\underline{\spc}_G)  \lar[shift left = 1, hook]
        \end{tikzcd}
    \end{center}
    Concretely, this implements group--completion pointwise, and this is what we mean by \textit{equivariant group--completion}.
\end{construction}

\begin{construction}[Forgetting norms]\label{cons:forgettingNorms}
    We explain here the compatibility of forgetting multiplicative norms with $G$--lax symmetric monoidal functors. First note that we have an adjunction $i \colon \ast \rightleftharpoons \orbit_G\op \: : p$ where $i$ is the inclusion of $G/G$, which is the initial object in $\orbit_G\op$. Hence, applying $\func(-,\cat_{\infty})$ we obtain an adjunction of $(\infty,2)$--categories $p^* \colon  \cat_{\infty}   \rightleftharpoons  \cat_{\infty, G} \: : i^*$.   Consequently, since $p^*(\sC) = \underline{\constant}_G(\sC) \coloneqq \sC\times \orbit_G\op$ and $i^*(\underline{\D}) = \D_{G} =\eval_{G/G}\underline{\D}$, we see that
    \begin{equation}\label{eqn:constantAdjunction}
        \func_G\big(\underline{\constant}_G(\sC), \underline{\D}\big) \simeq \func\big(\sC, {\D}_{G}\big)
    \end{equation}
    In particular, the fully faithful functor of 1--categories
    $\finite_* \rightarrow \finite_{*G}\coloneqq \func(BG,\finite_*)$ given by $n \mapsto \coprod^nG/G$ induces a $G$--functor $q\colon \underline{\constant}_G(\finite_*) \longrightarrow \underline{\finite}_{*}$ given by $(n, G/H) \mapsto \coprod^nH/H$. Therefore, for a $G$--symmetric monoidal category $\underline{\D}^{\underline{\otimes}}\in \cmonoid_G(\underline{\cat}_{\infty,G})$, we obtain the map 
    \[\quad\quad\calg_G(\underline{\D}^{\underline{\otimes}}) \coloneqq \func_{G/\underline{\finite}_*}\inerts(\underline{\finite}_{*}, \underline{\D}^{\underline{\otimes}}) \xlongrightarrow{q^*} \func\inerts_{/\finite_*}(\finite_*, \D_{G}^{\otimes}) =: \calg(\D_G^{\otimes})\]
    where, to analyse the target, we have  used that 
    \begin{equation*}
        \begin{split}
            \func_{G/\underline{\finite}_*}(\underline{\constant}_G(\finite_*) ,\underline{\D}^{\underline{\otimes}})&\coloneqq \func_G(\underline{\constant}_G(\finite_*),\underline{\D}^{\underline{\otimes}})\times_{\func_G(\underline{\constant}_G(\finite_*), \underline{\finite}_{*})} \{q\}\\
            &\simeq\func(\finite_*,\D_{G}^{\otimes})\times_{\func(\finite_*, {\finite}_{*G})} \{q\}\\
            &\simeq\func(\finite_*,\D_{G}^{\otimes})\times_{\func(\finite_*, {\finite}_{*})} \{q\} =: \func_{/\finite_*}(\finite_*, \D_{G}^{\otimes})
        \end{split}
    \end{equation*}
    where the second equivalence is by \cref{eqn:constantAdjunction} and the third since $q\in \func(\finite_*,\finite_{*G})$ lies in the full subcategory $\func(\finite_*,\finite_*)$ via the fully faithful functor $q\colon \finite_*\subseteq \finite_{*G}$. Intuitively, the functor $q^*$ forgets the norm structures on a $G$--$\mathbb{E}_{\infty}$--ring object in $\underline{\D}$ and so we will also denote it by $\forget$ in the sequel.

    All in all, as a consequence, for a $G$--lax symmetric monoidal functor $F\colon \underline{\sC}^{\underline{\otimes}}\rightarrow \underline{\D}^{\underline{\otimes}}$, since $F\colon \calg_G(\underline{\sC}^{\underline{\otimes}})\rightarrow \calg_G(\underline{\D}^{\underline{\otimes}})$ is given by postcomposition and the forgetful functor is given by precomposition along $q\colon \underline{\constant}_G(\finite_*) \rightarrow \underline{\finite}_{*}$, we obtain a commuting square
    \begin{center}
        \begin{tikzcd}
            \calg_G(\underline{\sC}^{\underline{\otimes}}) \rar["F"] \dar["\forget"]& \calg_G(\underline{\D}^{\underline{\otimes}})\dar["\forget"]\\
            \calg({\sC}_G^{{\otimes}})\rar["F"] & \calg({\D}_G^{{\otimes}})
        \end{tikzcd}
    \end{center}
\end{construction}

\section{The main theorem}\label{section:mainTheorem}

In order to state and prove the theorem, we will need a few more terminologies and observations.

\begin{notation}\label{nota:canonicalLocalisationComparsion}
    In this note, two kinds of ring localisations will feature and we define and relate them here. Let $R\in\calg(\spectra_G)$ and $\underline{S} = \{S_H\}_{H\leq G}$ be a $G$\textit{--subset} of the zeroth equivariant homotopy Mackey functor $\underline{\pi}_0R$ of $R$. That is, for any $H\leq G$, $\underline{S}$ satisfies $\res^G_HS_G\subseteq S_H\subseteq \pi_0^HR\coloneqq \pi_0R^H$. Now for any $A\in\calg(\spectra_G)$, we define 
    \[\map_{\calg(\spectra_G)}^{\underline{S}^{-1}}(R,A) \quad\quad\text{ and } \quad\quad \map_{\calg(\spectra_G)}^{S_G^{-1}}(R,A)\]
    to be subcomponents of $\map_{\calg(\spectra_G)}(R,A)$ of $\mathbb{E}_{\infty}$--algebra maps $R\rightarrow A$ which send elements in $\underline{S}$ to units in $\underline{\pi}_0A$ and send elements in $S_G$ to units in $\pi_0^GA$, respectively. By general theory (cf. \cite[Appen. A]{nikolaus} for example), we know that the latter mapping space is corepresented by a telescopic localisation $S_G^{-1}R$ of $R$ against elements in $S_G\subseteq \pi_0^GR$ (ie. $\map_{\calg(\spectra_G)}^{S_G^{-1}}(R,A)\simeq \map_{\calg(\spectra_G)}(S_G^{-1}R,A)$). In particular, we have that $\underline{\pi}_{\star}S_G^{-1}R\cong S_G^{-1}\underline{\pi}_{\star}R$.
    
    On the other hand, \textit{if} the former mapping space is corepresentable, then we will write the corepresenting object as $L_{\underline{S}^{-1}}R$. In general, this need not be given by a nice formula in terms of a telescopic localisation since we need to invert different sets of elements at different subgroups $H\leq G$ that do not all come from restricting elements from $S_G$ (ie. the inclusion $\res^G_HS_G\subseteq S_H$ might be proper), and so $\underline{\pi}_*L_{\underline{S}^{-1}}R$ need not admit a nice description as a Mackey functor with elements in $\underline{S}$ inverted. However, since maps $R\rightarrow A$ which invert $\underline{S}$ must necessarily invert $S_G$, we do have an inclusion
    \[\map_{\calg(\spectra_G)}^{\underline{S}^{-1}}(R,-) \xhookrightarrow{\phantom{\quad}} \map_{\calg(\spectra_G)}^{S_G^{-1}}(R,-)\]
    Thus, when $L_{\underline{S}^{-1}}R$ exists, this inclusion is induced by a canonical comparison map     in $\calg(\spectra_G)$
    \begin{equation}\label{eqn:canonicalComparsion}
        S_G^{-1}R \longrightarrow L_{\underline{S}^{-1}}R.
    \end{equation}
\end{notation}

\begin{notation}\label{nota:intro}
    Let $M\in\cmonoid(\spc_G)$. We write $\underline{\pi}_M\subseteq \underline{\pi}_0\sphere_G[M]$ for the image of the Hurewicz map on the equivariant homotopy groups $\underline{\pi}_0M \rightarrow\underline{\pi}_0{\Omega}^{\infty}_G\sphere_G[M] = \underline{\pi}_0\sphere_G[M]$ induced by the adjunction unit $\id\Rightarrow \loops_G\sphere_G$. This is clearly a $G$--subset in the sense defined above.

\end{notation}

We are now ready to state the main theorem of this note:

\begin{theorem}\label{mainThmA}
    Let $M\in\cmonoid(\spc_G)$ be an $\mathbb{E}_{\infty}$--monoid $G$--space.
    \begin{enumerate}
        \item[(i)] The object $L_{(\underline{\pi}_M)^{-1}}\sphere_G[M]$ exists and the group--completion map $M \rightarrow \Omega BM$ induces an equivalence in $\calg({\spectra}_G)$
    \[L_{(\underline{\pi}_M)^{-1}}\sphere_G[M] \xlongrightarrow{\simeq} \sphere_G[\Omega BM]\]
    \item[(ii)] Moreover, if $M$ additionally has the structure of a $G$--$\mathbb{E}_{\infty}$--monoid $G$--space - ie. $M\in\cmonoid_G({\spc}_G)$ - then $\sphere_G[\Omega BM]\simeq L_{(\underline{\pi}_M)^{-1}}\sphere_G[M] $ refines to a $G$--$\mathbb{E}_{\infty}$--ring object. In other words, it lifts to an object in $\calg_G({\spectra}_G)$. Furthermore, in this case, the canonical map from \cref{eqn:canonicalComparsion} 
    \[(\pi_M^G)^{-1}\sphere_G[M] \longrightarrow L_{(\underline{\pi}_M)^{-1}}\sphere_G[M]\simeq \sphere_G[\Omega BM] \]
    is an equivalence so that we have the expected localisation effect on homotopy groups, ie. $\underline{\pi}_{\star}\sphere_G[\Omega BM]\cong (\pi_M^G)^{-1}\underline{\pi}_{\star}\sphere_G[M]$.
    \end{enumerate}
\end{theorem}

We now turn to the proof of the first part of the theorem. We emphasise again that the theory of $G$--categories is not required in this part.

\begin{proof}[Proof of  \cref{mainThmA} (i)]
    The proof is exactly the same as that of \cite[Thm. 1]{nikolaus}. To wit, we first claim that $\Omega BM\in \cmonoid(\spc_G)$ satisfies the following universal property: for every $X\in\cmonoid(\spc_G)$, the map
    \[{\map}_{{\cmonoid}(\spc_G)}(\Omega BM,X)\rightarrow {\map}_{{\cmonoid}(\spc_G)}^{(\underline{\pi}_0M)^{-1}}(M,X)\]
    induced by $\eta\colon M\rightarrow \Omega BM$
    is an equivalence, where ${\map}^{(\underline{\pi}_0M)^{-1}}_{\cmonoid(\spc_G)}\subseteq {\map}_{\cmonoid(\spc_G)}$ means the subcomponents of maps where $\underline{\pi}_0M$ is sent to elements that admit additive inverses in $\underline{\pi}_0X$. The map lands in this subcomponent since $\Omega BM$ is group--complete. To prove the claim, define $X^{\times}$ as the pullback in $\cmonoid(\spc_G)$
    \begin{center}
        \begin{tikzcd}
            X^{\times} \rar["i"]\dar\ar[dr, phantom, very near start, "\lrcorner"] & X\dar\\
            (\underline{\pi}_0X)^{\times} \rar[hook] & \underline{\pi}_0X
        \end{tikzcd}
    \end{center}
    so that $X^{\times}\in\cgroup(\spc_G)$. Now consider the commuting diagram
    \begin{center}
        \begin{tikzcd}
            {\map}_{{\cmonoid}(\spc_G)}(\Omega BM,X)\rar["\eta^*"] & {\map}_{{\cmonoid}(\spc_G)}^{(\underline{\pi}_0M)^{-1}}(M,X)\\
            {\map}_{{\cmonoid}(\spc_G)}(\Omega BM,X^{\times})\uar["i_*"] \rar["\eta^*"'] & {\map}_{\cmonoid(\spc_G)}(M,X^{\times})\uar["i_*"'] 
        \end{tikzcd}
    \end{center}
    The left vertical $i_*$ is an equivalence since $\Omega BM$ is group--complete and $(-)^{\times}$ is the right adjoint to the inclusion $\cgroup(\spc_G)\subseteq \cmonoid(\spc_G)$; the bottom $\eta^*$ is an equivalence since $X^{\times}$ is group--complete and $\Omega BM$ is the group--completion of $M$ by \cref{lem:groupCompletion}; the right vertical $i_*$ is an equivalence because maps in $\map^{(\underline{\pi}_0M)^{-1}}$ are precisely those that land in $X^{\times}$ by definition. Therefore, all in all, the top horizontal $\eta^*$ is also an equivalence, as claimed.

    Now by the adjunction \cref{mainAdjunction}, for any $A\in \calg({\spectra}_G)$, we have
    \begin{equation}\label{eqn:mainComputation}
        \begin{split}
            {\map}_{{\calg}({\spectra}_G)}(\sphere_G[\Omega BM], A)&\simeq {\map}_{{\cmonoid}(\spc_G)}(\Omega BM, {\Omega}^{\infty}_GA)\\
            &\simeq {\map}_{{\cmonoid}(\spc_G)}^{(\underline{\pi}_0M)^{-1}}(M, {\Omega}^{\infty}_GA)\\
            &\simeq {\map}_{{\calg}({\spectra}_G)}^{(\underline{\pi}_M)^{-1}}(\sphere_G[M], A)
        \end{split}
    \end{equation}
    where the second equivalence is by the claim above. By \cref{nota:canonicalLocalisationComparsion},  $\sphere_G[\Omega BM]$ therefore computes $L_{(\underline{\pi}_M)^{-1}}\sphere_G[M]$, as desired.
\end{proof}

We now turn to the task of refining to normed structures when the input is more highly structured, ie. when $M\in \cmonoid_G(\underline{\spc}_G)$. Before that, it would be useful to formulate the following intermediate notion together with a couple of easy consequences which would help us identify the homotopy groups of the abstract localisation we have so far.

\begin{condition}[Torsion--extensions]\label{condition:torsionExtensions}
    Let $R\in\calg(\spectra_G)$ and $\underline{S}\subseteq \underline{\pi}_0R$ be a $G$--subset of the zeroth equivariant homotopy groups of $R$. We say that $\underline{S}$ satisfies the \textit{torsion--extension condition} if 
    for any $H\leq G$, the inclusion $\res^G_HS_G\subseteq S_H$ is a torsion--extension, ie. for any $a\in S_H$, there exists a $r\in \pi_0^HR$ such that $r\cdot a\in\res^G_HS_G$.
\end{condition}

\begin{remark}
    The reason for this choice of terminology was an analogy in the case of modules: if $I\subseteq J\subseteq R$ are $R$--submodules satisfying the analogous condition, then $J/I$ is a torsion $R$--module. In any case, the next three lemmas should clarify our interest in this condition.
\end{remark}

\begin{lemma}\label{lem:TorsionExtensionLocalisation}
    If $R\in\calg({\spectra}_G)$ and $\underline{S}\subseteq \underline{\pi}_0R$ is a multiplicatively closed $G$--subset satisfying \cref{condition:torsionExtensions}, then $L_{\underline{S}^{-1}}R$ exists and  the canonical map $S_G^{-1}R \longrightarrow L_{\underline{S}^{-1}}R$ from \cref{eqn:canonicalComparsion}  is an equivalence. Furthermore, in this case, for any $K\leq G$, we have that $\res^G_KS^{-1}_GR\simeq S^{-1}_K\res^G_KR$.
\end{lemma}
\begin{proof}
    As explained in \cref{nota:canonicalLocalisationComparsion}, the canonical map in the statement induces an inclusion of subcomponents $\map_{\calg(\spectra_G)}^{\underline{S}^{-1}}(R, A) \hookrightarrow\map_{\calg(\spectra_G)}^{S_G^{-1}}(R,A)$.  Hence all we have to do is to show that all components in the target are hit. So suppose $\varphi \colon R\rightarrow A$ inverts elements in $S_G$. We need to show that for all $H\leq G$, $\varphi|_H\colon \res^G_HR\rightarrow \res^G_HA$ sends elements in $S_H\subseteq \pi_0^HR$ to units in $\pi_0^HA$. 
    
    Thus, fix $H\leq G$ and let $a\in S_H$. By hypothesis, there exists an $r\in\pi_0^HR$ such that $r\cdot a \in \res^G_HS_G$. Since $\varphi|_H$ inverts $r\cdot a$, let $x\in\pi_0^HA$ such that $1 = x\cdot \varphi|_H(r\cdot a) = x\cdot \varphi|_H(r)\cdot \varphi|_H(a)$. In particular, since everything is commutative, $x\cdot \varphi|_H(r)$ is the inverse of $\varphi|_H(a)$, and so $\varphi|_H$ inverts $a$ too. Therefore, since $a$ was arbitrary, we see that $\varphi|_H$ must have inverted all of $S_H$ as required.

    For the last statement, first observe that $\res^G_KS^{-1}_GR\simeq (\res^G_KS_G)^{-1}\res^G_KR$. Hence, since $\res^G_KS_G\subseteq S_K\subseteq \pi_0^KR$, we see that a priori $\res^G_KS^{-1}_GR$ has inverted possibly fewer elements than has $S^{-1}_K\res^G_KR$. However, the same argument as in the previous paragraph shows that under our hypothesis on $R$, we indeed have $\res^G_KS^{-1}_GR\simeq S^{-1}_K\res^G_KR$ as wanted.
\end{proof}

\begin{lemma}\label{lem:normClosureTorsionExtension}
    Let $R\in\calg_G(\myuline{\spectra}_G)$ be a $G$--$\mathbb{E}_{\infty}$--ring object and $\underline{S}\subseteq \underline{\pi}_0R$ be a $G$--subset that is closed under the norms. Then $\underline{S}$ satisfies \cref{condition:torsionExtensions}.   
\end{lemma}
\begin{proof}
    Fix $H\leq G$ and let $a\in S_H$. We want to show that there is an $r\in\pi_0^HR$ such that $r\cdot a\in \res^G_HS_G$. For this, consider $\norm^G_Ha\in \pi_0^GR$ which is in fact in $S_G\subseteq \pi_0^GR$ by the norm--closure hypothesis. Then by the norm double coset formula, we get
    \[\res^G_H\norm^G_Ha = \prod_{g\in H\backslash G/H} \norm^H_{H^g\cap H}g_*\res^H_{H\cap H^g}a \in \res^G_HS_G\]
    where $a$ is a factor on the right (ie. when $g=e$), whence the claim.
\end{proof}

\begin{lemma}\label{lem:normClosureMonoidRing}
    If $M\in\cmonoid_G(\underline{\spc}_G)$, then $\underline{\pi}_M\subseteq \underline{\pi}_0\sphere_G[M]$ is closed under the norms.
\end{lemma}
\begin{proof}
    First of all, by \cref{lem:basicAdjunctionCMonNalg} we know $\sphere_G[M]$ refines to a $G$--$\mathbb{E}_{\infty}$--ring object. Now fix $H\leq G$ and suppose we have $n\in\pi_0^HM$ with associated element $\overline{n}\in \pi_0^H\sphere_G[M]$. Thus by definition the normed element $\norm^G_H\overline{n}\in \pi_0^G\sphere_G[M]$ is given by 
    \[\sphere_G=\norm^G_H\sphere_H\xlongrightarrow{\norm^G_H\overline{n}} \norm^G_H\res^G_H\sphere_G[M] \simeq\sphere_G[\prod_{G/H}\res^G_HM]\xlongrightarrow{\sphere[\oplus_{G/H}]} \sphere_G[M]\] 
    Here $\oplus_{G/H}\colon \prod_{G/H}\res^G_HM\rightarrow M$ is the $G$--semiadditivity adjunction counit of an object $M\in\cmonoid_G(\underline{\spc}_G)$. The middle equivalence is since $\sphere_G[\prod_{G/H}-]\simeq \norm^G_H\sphere_H[-]$ from the $G$--symmetric monoidality of the functor $\sphere_G[-]$ from \cref{lem:basicAdjunctionCMonNalg}
    
    Now,  the natural transformation $(-) \Rightarrow \Omega^{\infty}_G\sphere_G(-)$ from \cref{lem:basicAdjunctionCMonNalg} together with the adjunction counit $\prod^G_H\res^G_HM\xrightarrow{\oplus_{G/H}} M$  yield the commuting diagram
    \begin{center}
        \begin{tikzcd}
            \ast \dar["n"']\rar& \loops_G\sphere_G\dar["\loops_G\norm^G_H\overline{n}"]\\
            \prod^G_H\res^G_HM \dar["\oplus_{G/H}"']\rar & \Omega^{\infty}_G\sphere_G[\prod^G_H\res^G_HM] \simeq \loops_G\norm^G_H\res^G_H\sphere_G[M]\dar["\Omega^{\infty}_G\sphere_G{[}\oplus_{G/H}{]}"]\\
            M\rar & \Omega^{\infty}_G\sphere_G[M]
        \end{tikzcd}
    \end{center}
    This implies that the normed up element $\norm^G_H\overline{n}\in \pi_0^G\sphere_G[M]$ already came from the element $\oplus_{G/H}n\in \pi_0^GM$ and so $\underline{\pi}_M\subseteq \underline{\pi}_0\sphere_G[M]$ is closed under norms.
\end{proof}

We now cash in all the work we have done to complete the proof of the theorem.

\begin{proof}[Proof of   \cref{mainThmA} (ii)]
    The main point is that we have commuting squares
    \begin{center}
        \begin{tikzcd}
            \cmonoid_G(\underline{\spc}_G) \rar[shift left =1, "\sphere_G{[}-{]}"] \dar["\forget"']& \calg_G(\myuline{\spectra}_G^{\underline{\otimes}})\lar[shift left = 1, "\loops_G"] \dar["\forget"]\\
            \cmonoid({\spc}_G) \rar[shift left =1, "\sphere_G{[}-{]}"] & \calg({\spectra}_G^{\otimes})\lar[shift left = 1, "\loops_G"] \\
        \end{tikzcd}
    \end{center}
    obtained by using the $G$--lax symmetric monoidality of the adjunction $\sphere_G[-]\dashv \loops_G$ from \cref{lem:basicAdjunctionCMonNalg} together with the commuting square in \cref{cons:forgettingNorms}. In fact, for this proof we only need that the $(\sphere_G[-],\forget)$ square commutes. Thus, if $M\in \cmonoid_G(\underline{\spc}_G)$ so that $\Omega BM$ is again in $\cmonoid_G(\underline{\spc}_G)$ by \cref{lem:groupCompletion}, then $\sphere_G[\Omega BM]$ - which is equivalent to $L_{(\underline{\pi}_M)^{-1}}\sphere_G[M]$ by part (i) of the theorem - naturally refines to the structure of an object in $\calg_G(\myuline{\spectra}_G^{\underline{\otimes}})$, ie. it canonically attains the multiplicative norms.  The final statement of part (ii) is then a direct combination of \cref{lem:TorsionExtensionLocalisation,lem:normClosureTorsionExtension,lem:normClosureMonoidRing}.
\end{proof}

\begin{remark}
    The norm closure of the subset $\underline{\pi}_M\subseteq \underline{\pi}_0\sphere_G[M]$ from \cref{lem:normClosureMonoidRing} should have indicated why the localisation $(\underline{\pi}_M)^{-1}\sphere_G[M]$ even had a chance of attaining the multiplicative norms. In general, a localisation on a $G$--$\mathbb{E}_{\infty}$--ring need not refine again to a $G$--$\mathbb{E}_{\infty}$--ring, as is well documented for instance in \cite{hillMultiplicativeLocalisation}. Nonetheless, the norm closure of a multiplicative subset is a necessary and sufficient property for the localisation to refine to the structure of a $G$--$\mathbb{E}_{\infty}$--ring. This can be deduced for example from \cite[Lem. 5.27]{shahQuigley}.
\end{remark}

Finally, we use \cref{mainThmA} to quickly deduce \cref{thm:Bredon}.

\begin{proof}[Proof of \cref{thm:Bredon}]
    Let $K\leq G$ and $\underline{N}$ a $G$--Mackey functor, thought of as an Eilenberg--Mac Lane genuine $G$--spectrum (see for example \cite[Ex. 4.41]{schwede}). Then by definition of $RO(G)$--graded Bredon homology (\textit{loc. cit.}), we have 
    $H_{\star}^K(\Omega BM; \underline{N}) = \pi_{\star}^K\big(\underline{N}\otimes \sphere_G[\Omega BM]\big)$. Moreover, by the second part of \cref{lem:TorsionExtensionLocalisation}, we know that $\res^G_K(\pi^G_M)^{-1}\sphere_G[M]\simeq (\pi^K_M)^{-1}\res^G_K\sphere_G[M]$ and so we get
    \[H_{\star}^K(M; \underline{N})[(\pi_0^KM)^{-1}] = (\pi^K_M)^{-1}\pi^K_{\star}\big(\underline{N}\otimes \sphere_G[M]\big) \cong\pi^K_{\star}\big(\underline{N}\otimes (\pi^G_M)^{-1}\sphere_G[M]\big)\]
    whence the result by \cref{mainThmA}.
\end{proof}

\section{Final remarks}\label{section:finalRemarks}

In this last section, we will comment on three points:
\begin{itemize}
    \item We analyse the geometric fixed points of the abstract localisation from \cref{nota:canonicalLocalisationComparsion} and show that it has an easy description,
    \item we explain a generic situation where the theorem might be applied,
    \item and we give a plentiful source of examples of $G$--$\mathbb{E}_{\infty}$--monoid $G$--spaces.
\end{itemize}

For the first point, as we have remarked in \cref{nota:canonicalLocalisationComparsion}, the abstract localisation $L_{\underline{S}^{-1}}R$, if it exists, has no reason to have a nice description in general. Notwithstanding, it does interact well with the geometric fixed points, as we now explain.

\begin{observation}\label{obs:geomFixPointsInvert}
    Let $\underline{S}\subseteq \underline{\pi}_0R$ be a multiplicative $G$--subset for some $R\in\calg(\spectra_G)$. Recall for instance from \cite[Cons. 6.10, Thm. 6.11]{mathewNaumannNoelI} that we have a lax symmetric monoidal Bousfield localisation $\Phi^G \colon \spectra_G \rightleftharpoons \spectra \: : \Xi^G$    which then induces a Bousfield localisation $\Phi^G \colon \calg(\spectra_G) \rightleftharpoons \calg(\spectra) \: : \Xi^G$. Here for $X\in\spectra$, $\Xi^GX$ is the $G$--spectrum such that $(\Xi^GX)^G\simeq X$ and $(\Xi^GX)^H\simeq 0$ for $H\lneq G$. Classically, this is also written as $\widetilde{E\mathcal{P}}\otimes X$ where $\mathcal{P}$ is the proper family of subgroups of $G$. We claim that the resulting equivalence $\map_{\calg(\spectra_G)}(R,\Xi^GA)\simeq \map_{\calg(\spectra)}(\Phi^GR,A)$ restricts to an equivalence
    \[\map_{\calg(\spectra_G)}^{\underline{S}^{-1}}(R,\Xi^GA)\simeq \map_{\calg(\spectra)}^{(\Phi^GS_G)^{-1}}(\Phi^GR,A)\]
    To see this, since $\Phi^G\Xi^G\simeq \id$, we know $\Phi^G$ induces an inclusion
    \[\map_{\calg(\spectra_G)}^{\underline{S}^{-1}}(R,\Xi^GA) \xhookrightarrow{\phantom{\quad}}\map_{\calg(\spectra)}^{(\Phi^GS_G)^{-1}}(\Phi^GR,A)\]
    To see that this is even an equivalence, suppose we have $\varphi \colon \Phi^GR\rightarrow A$ which inverts $\Phi^GS_G\subseteq \pi_0\Phi^GR$. The adjoint $\overline{\varphi} \colon R\rightarrow \Xi^GA$ is given by the composite
    \[\overline{\varphi} \colon R \xlongrightarrow{\eta} \Xi^G\Phi^GR \xlongrightarrow{\Xi^G\varphi} \Xi^GA\]
    where the adjunction unit $\eta$ is a map of $\mathbb{E}_{\infty}$--rings and sends elements in $S_G$ to elements in $\Phi^GS_G$. Therefore, $\overline{\varphi}$ must invert all elements in $S_G$. Moreover, since for $H\lneq G$, $(\Xi^GA)^H$ are equivalent to the zero rings, the maps $\res^G_H\overline{\varphi} \colon \res^G_HR \rightarrow \res^G_H\Xi^GA\simeq 0$ send everything to units for trivial reasons, and so in total $\overline{\varphi}$ indeed inverts elements in $\underline{S}$ as was to be shown.  
\end{observation}

\begin{proposition}\label{prop:geomFixPointsLocalisation}
    Let $R\in\calg(\spectra_G)$, $\underline{S}\subseteq \underline{\pi}_0R$  a multiplicative subset, and suppose $L_{\underline{S}^{-1}}R$ exists. Then the canonical map $\Phi^GR\rightarrow \Phi^GL_{\underline{S}^{-1}}R$ induces an equivalence $(\Phi^GS_G)^{-1}\Phi^GR\simeq \Phi^GL_{\underline{S}^{-1}}R$.
\end{proposition}
\begin{proof}
Let $A\in\calg(\spectra)$. Then
    \begin{equation*}
        \begin{split}
            \map_{\calg(\spectra)}(\Phi^GL_{\underline{S}^{-1}}R,A) & \simeq \map_{\calg(\spectra_G)}(L_{\underline{S}^{-1}}R, \Xi^G A)\\
            &\simeq \map_{\calg(\spectra_G)}^{\underline{S}^{-1}}(R,\Xi^GA)\\
            &\simeq \map_{\calg(\spectra)}^{(\Phi^GS_G)^{-1}}(\Phi^GR,A)\\
            &\simeq \map_{\calg(\spectra)}((\Phi^GS_G)^{-1}\Phi^GR,A)\\
        \end{split}
    \end{equation*}
    where  the third equivalence is by \cref{obs:geomFixPointsInvert}.
\end{proof}

\begin{remark}\label{example:didactic}
    Let $M\in\cmonoid(\spc_G)$. We claim that $\Phi^G\pi_M^G = \pi_{M^G}\subseteq \pi_0\sphere[M^G]$ where $\pi_{M^G}$ is the image of the nonequivariant Hurewicz map $\pi_0M^G \rightarrow \pi_0\sphere[M^G]$. Given this, we see by \cref{prop:geomFixPointsLocalisation} that 
    \[\Phi^G(\underline{\pi}_M^{-1}\sphere_G[M]) \simeq (\pi_{M^G})^{-1}\sphere[M^G]\]
    and so applying $\Phi^G$ reduces \cref{mainThmA} (i) to the classical group--completion theorem formulated for example in \cite[Thm. 1]{nikolaus}. To prove the claim, we want to show that the inclusion $\Phi^G\pi_M^G\subseteq \pi_{M^G}$ is surjective. By one of the defining properties of $\Phi^G$, we know that there is a commuting square
    \begin{center}
        \begin{tikzcd}
            \spc_G \rar["\sphere_G{[}-{]}"]\dar["(-)^G"'] & \spectra_G\dar["\Phi^G"]\\
            \spc \rar["\sphere"] & \spectra
        \end{tikzcd}
    \end{center}
    which yields a commuting square
    \begin{center}
    \footnotesize
        \begin{tikzcd}
            \pi_0^GM=\pi_0\map_{\spc_G}(\ast, M) \rar["\sphere_G{[}-{]}"]\dar["\cong"] &  \pi_0^G\sphere_G[M] = \pi_0\map_{{\spectra_G}}\big(\sphere_G, \sphere_G[M]\big)\dar["\Phi^G"]\\
            \pi_0(M^G)=\pi_0\map_{\spc}(\ast, M^G) \rar["\sphere{[}-{]}"] \rar & \pi_0\sphere[M^G] = \pi_0\map_{{\spectra}}\big(\sphere, \sphere[M^G]\big)
        \end{tikzcd}
    \end{center}
    \normalsize
    This implies that $\Phi^G\pi_M^G\subseteq \pi_{M^G}$ is surjective, as desired.
\end{remark}

Next, we turn to the matter of recording a generic toy situation where our theorem might be useful. This manoeuvre is an immediate generalisation of its (standard) nonequivariant analogue.

\begin{proposition}[Equivariant simply--connected homology Whitehead theorem]
    Suppose we have a map $f\colon X\rightarrow Y $ of $G$--spaces which induces an equivalence $\sphere_Gf\colon \sphere_G[X]\rightarrow \sphere_G[Y]$. Suppose moreover that $X, Y$ are both $G$--simply--connected (ie. $X^H$ and $Y^H$ are simply--connected for all $H\leq G$). Then the map $f\colon X\rightarrow Y$ was already a $G$--equivalence. 
\end{proposition}
\begin{proof}
    To see this, we need to show that we have an equivalence for all fixed points. So let $H\leq G$. Applying the $H$--geometric fixed points $\Phi^H$ to the equivalence $\sphere_Gf$ gives us an equivalence $\Phi^H\sphere_Gf\simeq \sphere[f^H]\colon \sphere[X^H]\xlongrightarrow{\simeq} \sphere[Y^H]$. Hence, by the ordinary simply--connected homology Whitehead theorem, the map of spaces  $f^H \colon X^H\rightarrow Y^H$ is an equivalence, as was to be shown.
\end{proof}

Our \cref{mainThmA} can then potentially be used in conjunction with this in the following way. Suppose we have a map of $G$--$\mathbb{E}_{\infty}$--monoids $N\rightarrow\Omega BM$ where we already understand $\sphere_G[N]$ and where $\Omega BM$ and $N$ are $G$--simply--connected. Since the theorem gives a formula for $\sphere_G[\Omega BM]$, we might be able to use it to show that $\sphere_G[N]\rightarrow\sphere_G[\Omega BM]$ is an equivalence. If this were true, then by the equivariant Whitehead proposition above, we can deduce that $N\rightarrow\Omega BM$ is an equivalence, thus giving a computation of $\Omega BM$ in terms of $N$.

Of course, this toy situation might not be so applicable since $G$--simply connectedness is an unreasonable condition to demand in general. Our intention for this was only to indicate a template over which other variations might be beneficial in specific circumstances.

Lastly, we end the main body of this note by recording a huge standard source of potentially interesting examples of $G$--$\mathbb{E}_{\infty}$--monoid $G$--spaces to consider.

\begin{example}\label{example:GMonoidSpaces}
    $G$--$\mathbb{E}_{\infty}$--monoid $G$--spaces, for which the localisation formula of \cref{mainThmA} (ii) holds, are in abundant supply. One fertile source is small semiadditive $\infty$--categories (which include stable $\infty$--categories) equipped with $G$--actions, ie. objects in $\func(BG,\cat^{\oplus}_{\infty})$. If $\sC$ were one such instance, then $\{\sC^{hH}\}_{H\leq G}$ assembles to a $G$--$\mathbb{E}_{\infty}$--monoid $G$--category. In other words, it is an object in $\mackey_G(\cat_{\infty}^{\oplus})$ (cf.  \cite[$\S8$]{barwick2} for an explanation of this). Then taking the groupoid core yields a $G$--$\mathbb{E}_{\infty}$--monoid $G$--space $\{(\sC^{hH})^{\simeq}\}_{H\leq G}\in \mackey_G(\spc)\simeq \cmonoid_G(\underline{\spc}_G)$. In fact, this procedure of producing $G$--$\mathbb{E}_{\infty}$--monoid $G$--spaces by taking groupoid cores works more generally for any $G$--semiadditive $G$--category.
    
    Concrete examples belonging to this template include equipping the trivial $G$--action on categories like finitely generated projective $R$--modules $\mathrm{Proj}_R$ for $R\in\mathrm{CRing}$ or perfect $A$--modules $\mathrm{Perf}_A$ for $A\in\calg(\spectra)$. These yield the objects $\{\map(BH, \mathrm{Proj}_R^{\simeq})\}_{H\leq G}$ and $\{\map(BH, \mathrm{Perf}_A^{\simeq})\}_{H\leq G}$ in $\cmonoid_G(\underline{\spc}_G)$, the group--completions of which give the so--called \textit{Swan} equivariant K--theories. Familiar examples of $G$--spectra obtained in this manner include $\mathrm{ku}_G$ and $\mathrm{ko}_G$. Another interesting source of semiadditive categories equipped with $G$--actions come from finite Galois extensions of fields $K\subseteq L$. In this case, the $G$--Galois action on $\mathrm{Vect}_L^{\mathrm{fd}}$ yields the $G$--$\mathbb{E}_{\infty}$--monoid $G$--space
    $\{(\mathrm{Vect}_{L^H}^{\mathrm{fd}})^{\simeq}\}_{H\leq G}$.
\end{example}

\appendix

\section{Algebras in equivariant cartesian symmetric monoidal structures}\label{appendix}

We will provide in this appendix a proof of the folklore result that $\cmonoid_G(\underline{\sC})\simeq \calg_G(\underline{\sC}\ttimes)$, which heuristic intuition we explain at the end of \cref{setting:GCategories}. The proof will be a straightforward - if a bit tedious - adaptation of the proof by Lurie from \cite[Prop. 2.4.2.5]{lurieHA} as organised by Chu and Haugseng \cite{chuHaugsengII} in the language of so--called \textit{cartesian patterns}. The main idea of Lurie's proof is that there is a nice model for the cartesian symmetric monoidal structure which embeds inside a larger category which in turn admits a convenient universal property of being mapped into. In the interest of space and as this is a necessarily technical result, we will assume some familiarity with the formalism and underpinnings of parametrised homotopy theory (cf. \cite{shahThesis, nardinExposeIV}), as well as the associated factorisation system and operad theory as laid out in \cite[\S3, \S4]{shahPaperII} and \cite[\S2.1-\S2.3]{nardinShah}. We will however provide basic recollections and precise references for the sake of comprehensibility. Lastly, we should also mention that this is an extremely brisk and minimalistic account sufficient for our purposes, and it might be interesting to investigate the notion of parametrised cartesian patterns along the level of generality in \cite{chuHaugsengII}. 

Our first order of business is to set up the basic theory of $G$--cartesian patterns and their associated monoids.

\begin{notation}
    It will be convenient to adopt the following conventions to lighten our notational load: for $\underline{\sC}$ a $G$--category, $H\leq G$, and a $H$--object $X\in\sC_H$ (which can equivalently be viewed as a $H$--functor $\underline{\ast}\rightarrow\res^G_H\underline{\sC}$),
    \begin{itemize}
        \item we write $\underline{\sC}_H$ for the $H$--category $\res^G_H\underline{\sC}$. Note that this does \textit{not} conflict with the notation $\sC_H$ from \cref{setting:GCategories}. As such, we will also write $H$--objects as $X\in\underline{\sC}_H$,
        \item we will write $\underline{\sC}_{X/}$ to mean the $H$--category $(\underline{\sC}_H)_{X/}$,
        \item for a $G$--functor $\underline{\D}\rightarrow\underline{\sC}$, we will write $\underline{\D}_X$ for the $H$--category $\underline{\ast}\times_{\underline{\sC}_H}\underline{\D}_H$, where $\underline{\ast}\rightarrow \underline{\sC}_H$ is the $H$--functor picking out $X$.
    \end{itemize}
    For a map $f\colon V\rightarrow W$ in $\orbit_G$, we write $f^*\colon \sC_W\rightarrow \sC_V$ for the restriction functor. If they exist, we write $f_!, f_*\colon \sC_V\rightarrow\sC_W$ for the indexed coproduct and indexed product associated to $f$, respectively (note that in \cite{nardinExposeIV} the notations for $f_!, f_*$ are given respectively by $\coprod_f, \prod_f$). In this case, we have $f_!\dashv f^*\dashv f_*$.
\end{notation}

\begin{definition}[``{\cite[Def. 2.1]{chuHaugsengII}}'']
    Let $\underline{\orbit}\in\cat_G$. A $G$\textit{--algebraic pattern structure on $\underline{\orbit}$} is a $G$--factorisation system (that is, a fibrewise factorisation system closed under the restriction functors, cf. \cite[Def. 3.1]{shahPaperII}) on $\underline{\orbit}$ together with a collection of objects which are termed \textit{elementary objects}. We term the left (resp. right) class as the fibrewise inert (resp. fibrewise active) morphisms. A morphism of $G$--algebraic patterns is a $G$--functor $\underline{\orbit}\rightarrow\underline{\proper}$ which preserves the fibrewise inert and active morphisms as well as the elementary objects. Write $\underline{\orbit}\inerts$ for the subcategory of  $\underline{\orbit}$ containing only the fibrewise inert morphisms, and write $\underline{\orbit}\elementary\subseteq \underline{\orbit}\inerts$ for the full subcategory of elementary objects and fibrewise inert morphisms.
\end{definition}

\begin{notation}
    Fix $H\leq G$ and $O\in\underline{\orbit}_H$ a $H$--object. Write $\underline{\orbit}_{O/}\elementary\coloneqq \underline{\orbit}\elementary\times_{\underline{\orbit}\inerts}\underline{\orbit}_{O/}\inerts$
    for the category of fibrewise inert maps from $O$ to elementary objects, and morphisms are fibrewise inert maps between these.
\end{notation}

\begin{notation}
    We will follow Chu and Haugseng's notation from \cite{chuHaugsengI,chuHaugsengII} and use $\rightarrowtail$ to denote inert maps and $\rightsquigarrow$ to denote active maps.
\end{notation}

\begin{example}
    Recall from \cite[Def. 2.1.2]{nardinShah} the $G$--category of finite pointed $G$--sets $\underline{\finite}_*$. This is given at level $H$ by $\func(BH,\finite_*)$. It can also be described explicitly as follows: the objects in level $H$ for some $H\leq G$ look like $[U\rightarrow G/H]$  where $U$ is a finite $G$--set, and a morphism from $[U\rightarrow G/H]$ to $[W\rightarrow G/K]$ looks like
    \begin{center}\scriptsize
        \begin{tikzcd}
            U\dar & Z \rar\lar \dar & W\dar\\
            G/H & G/K \lar\rar[equal] & G/K
        \end{tikzcd}
    \end{center}
    where all maps in sight are $G$--equivariant and the induced map $Z\rightarrow U\times_{G/H}G/K$ is a summand inclusion.

    This is the prime example of a $G$--algebraic pattern, using that algebraic patterns are closed under limits in $\cat_{\infty}$ by \cite[Cor. 5.5]{chuHaugsengI}. Concretely, when $K=H$, the fibrewise inert maps are the ones where $Z\rightarrow W$ is an equivalence, and the fibrewise active maps are those where the induced map $Z\rightarrow U\times_{G/H}G/K$ is an equivalence (see \cite[Def. 2.1.3]{nardinShah}); the elementary objects are the objects $[G/H\xrightarrow{=}G/H]$ at level $H $ for each $H\leq G$.
\end{example}

Following \cite[Def. 6.1]{chuHaugsengI}, we may make the following:

\begin{definition}\label{defn:uniqueActiveLifting}
    We say that a morphism $f\colon \underline{\orbit}\rightarrow\underline{\proper}$ of $G$--algebraic patterns has \textit{unique lifting of fibrewise active morphisms } if for every $H\leq G$ and fibrewise active morphism $\phi\colon P \rightarrow f(O)$ in ${\proper}_H$, the space of lifts of $\phi$ to a fibrewise active morphism $O'\rightarrow O$ in ${\orbit}_H$ is contractible.
\end{definition}

Since $G$--coinitiality is a fibrewise statement by the dual of \cite[Thm. 6.7]{shahThesis} and \cref{defn:uniqueActiveLifting} is also fibrewise, we may deduce immediately from \cite[Lem. 6.2]{chuHaugsengI} the following:

\begin{lemma}\label{lem:coinitialCharacterisationACtiveLifting}
    A morphism of $G$--algebraic patterns $f\colon \underline{\orbit}\rightarrow\underline{\proper}$ has unique lifting of fibrewise active morphisms if and only if for every $H\leq G$ and all $P\in{\proper}_H$, the functor $\underline{\orbit}\inerts_{P/}\rightarrow \underline{\orbit}_{P/}$ is $G$--coinitial. 
\end{lemma}

\begin{definition}[``{\cite[Def. 2.6]{chuHaugsengII}}'']
    A $G$\textit{--cartesian pattern} is a $G$--algebraic pattern $\underline{\orbit}$ equipped with a morphism of $G$--algebraic patterns $|-|\colon \underline{\orbit}\rightarrow \underline{\finite}_*$ such that for every $H\leq G$ and $O\in\underline{\orbit}_H$, the induced map $\underline{\orbit}\elementary_{O/}\rightarrow \underline{\finite}_{*,|O|/}\elementary$
    is an equivalence. A \textit{morphism of $G$--cartesian patterns } is a morphism of $G$--algebraic patterns over $\underline{\finite}_*$.
\end{definition}

\begin{construction}
    Recall the \textit{characteristic morphisms} from \cite[Defn. 2.1.6]{nardinShah}, i.e. maps in $\underline{\finite}_*$ that look like
    \begin{center}\small
        \begin{tikzcd}
            U\dar & W \rar[equal]\lar \dar[equal] & W\dar[equal]\\
            G/H & W \lar["w"']\rar[equal] & W
        \end{tikzcd}
    \end{center}
    where $W=G/K$ is a $G$--orbit in $U$. We write such maps as $\chi_{[W\subseteq U]}$. These should be thought of as the analogue of the Segal maps $\rho^i$ (cf. \cite[Nota. 2.0.0.2]{lurieHA}) in the parametrised setting. For any $G$--functor $F\colon\underline{\finite}_*\rightarrow\underline{\sC}$ where $\underline{\sC}$ has finite indexed products, the map $\chi_{[W\subseteq U]}$ induces a canonical map of $H$--objects in $\underline{\sC}$
    \[F([U\rightarrow G/H])\longrightarrow w_*F([W\xrightarrow{=} W])\]
    since we have 
    \begin{equation*}
        \begin{split}
            F(\chi_{[W\subseteq U]}) \in \map_{\underline{\sC}}\Big(F([U\rightarrow G/H]), F([W\xrightarrow{=} W])\Big) &\simeq \map_{\sC_K}\Big(w^*F([U\rightarrow G/H]), F([W\xrightarrow{=} W])\Big)\\
            &\simeq \map_{\sC_H}\Big(F([U\rightarrow G/H]), w_*F([W\xrightarrow{=} W])\Big)
        \end{split}
    \end{equation*}
    by the fact that $\underline{\sC}$ has indexed products.
\end{construction}

\begin{definition}[``{\cite[Def. 2.9]{chuHaugsengII}}'']\label{defn:OMonoids}
    Let $\underline{\orbit}$ be a $G$--cartesian pattern and suppose $\underline{\sC}$ has finite indexed products. A $G$--functor $F\colon \underline{\orbit}\rightarrow\underline{\sC}$ is said to be an $\underline{\orbit}$\textit{--monoid} if for every $[U\rightarrow G/H]\in \underline{\finite}_{*H}$ and $O\in\underline{\orbit}_H$ lying over $[U\rightarrow G/H]$, writing $U = \coprod_{j=1}^nU_j$  for the $G$--orbital decomposition, $u_j\colon U_j\rightarrow G/H$ for the structure maps, and $\chi_{[U_j\subseteq U]}\colon O \rightarrow O_j$ with $O_j$ lying over $[U_j=U_j]$ afforded by the equivalence $\underline{\orbit}\elementary_{O/}\xrightarrow{\simeq} \underline{\finite}_{*,|O|/}\elementary$, the canonical map of $H$--objects in $\underline{\sC}$
    \[F(O) \longrightarrow \prod_{j=1}^nu_{j*}F(O_j)\]
    is an equivalence. By the $G$--cartesian pattern condition, this is equivalent to the following: writing $j\colon \underline{\orbit}\elementary\hookrightarrow \underline{\orbit}\inerts$ for the inclusion,  $F$ is an $\underline{\orbit}$--monoid if and only if the canonical map $F|_{\underline{\orbit}\inerts}\rightarrow j_*j^*(F|_{\underline{\orbit}\inerts})$    is an equivalence. We write $\monoid_{\underline{\orbit}}(\underline{\sC})\subseteq \func_G(\underline{\orbit},\underline{\sC})$ for the full subcategory of $\underline{\orbit}$--monoids in $\underline{\sC}$.
\end{definition}

\begin{remark}
    In the case $\underline{\orbit}=\underline{\finite}_*$, by an easy comparison of definitions with \cite[Def. 5.9]{nardinExposeIV}, we get that $\monoid_{\underline{\finite}_*}(\underline{\sC})\simeq \cmonoid_G(\underline{\sC})$ where $\cmonoid_G(\underline{\sC})$ is in the sense discussed in the body of the paper.
\end{remark}

The exact same argument as in \cite[Prop. 6.3]{chuHaugsengI}, which uses only formalities about Kan extensions such as fully faithfulness of Kan extensions along fully faithful functors \cite[Prop. 10.6]{shahThesis} as well as \cref{lem:coinitialCharacterisationACtiveLifting}, applies here to yield the following:

\begin{lemma}\label{lem:activeLiftMonoidRightKanExtend}
    If $f\colon \underline{\orbit}\rightarrow\underline{\proper}$ is a morphism of $G$--algebraic patterns that has unique fibrewise active lifting, then the right Kan extension $f_*\colon {\func}_G(\underline{\orbit},\underline{\sC})\rightarrow{\func}_G(\underline{\proper},\underline{\sC})$ restricts to $f_*\colon \monoid_{\underline{\orbit}}(\underline{\sC})\rightarrow\monoid_{\underline{\proper}}(\underline{\sC})$.
\end{lemma}

Next, we work towards constructing the equivariant generalisation of Lurie's model \cite[Prop. 2.4.1.5]{lurieHA} for the cartesian symmetric monoidal structure for a $G$--category with finite indexed products.

\begin{construction}\label{cons:GammaTimesConstruction}
    Let $\underline{\Gamma}\ttimes$ denote the full subcategory of $\underline{\finite}_*^{\Delta^1}$ spanned by the fibrewise inert morphisms. It will be convenient to denote by $[U\rightarrowtail V]_{G/H} \in \underline{\Gamma}\ttimes_H$ the $H$--object
    \begin{center}\scriptsize
        \begin{tikzcd}
            U\dar & Z\lar\rar["\simeq"]\dar& V \dar\\
            G/H \rar[equal]& G/H \rar[equal]& G/H         
        \end{tikzcd}
    \end{center}
    By \cite[Prop. 3.5 (1)]{shahPaperII}, we know that $\eval_0\colon \underline{\Gamma}\ttimes\longrightarrow \underline{\finite}_*$ is a $G$--cartesian fibration. By the proof of said result, we see that for a $H$--map $f\colon U \rightarrow V$ in $\underline{\finite}_*$, the $H$--functor $f^!\colon \underline{\Gamma}\ttimes_{[V\rightarrow G/H]}\rightarrow \underline{\Gamma}\ttimes_{[U\rightarrow G/H]}$ associated to the $G$--cartesian fibration is given concretely as follows: for $K\leq H$ and $[V\rightarrowtail W]_{G/K}$ a $K$--object in $\underline{\Gamma}\ttimes_{[V\rightarrow G/H]}$, the $K$--object $f^!([V\rightarrowtail W]_{G/K})\in\underline{\Gamma}\ttimes_{[U\rightarrow G/H]}$ is given by the unique dashed fibrewise inert map in 
    \begin{equation}\label{eqn:inertActiveFactorisationPreimageFunctor}
        \begin{tikzcd}
            U \dar[dashed, tail]\rar["f"] & V \dar[tail]\\
            W' \rar[rightsquigarrow] & W
        \end{tikzcd}
    \end{equation}
    obtained by virtue of the unique fibrewise inert--active factorisation.
    
    Observe also that for $[U \rightarrow G/H]\in\underline{\finite}_{*H}$, $\underline{\Gamma}\ttimes_{[U \rightarrow G/H]}$ is a $H$--category such that for any $K\leq H$, the fibre over $H/K$ is given by the opposite of the poset of $K$--subsets of the $H$--set $U$ (compare with \cite[Cons. 2.4.1.4]{lurieHA}): this is because the fibrewise inert maps pick out the orbits in $U$. 
\end{construction}

\begin{construction}[``{\cite[Cons. 2.4.1.4]{lurieHA}}'']\label{cons:overlineCartesianSymmetricMonoidalStructure}
    Applying \cite[Thm. 9.3 (2)]{shahThesis} or \cite[Recoll. 4.3]{shahPaperII} to the $G$--cartesian fibration $\eval_0\colon \underline{\Gamma}\ttimes\longrightarrow \underline{\finite}_*$ and the $G$--cocartesian fibration $\underline{\sC}\times \underline{\finite}_*\rightarrow \underline{\finite}_*$ we obtain a $G$--cocartesian fibration $\overline{\underline{\sC}}\ttimes\rightarrow\underline{\finite}_*$. By \cite[Thm. 4.9]{shahPaperII}, this construction satisfies a universal property which implies in particular that 
    \begin{equation}\label{eqn:universalPropertyOverlineCartesian}
        \underline{\func}_{/\underline{\finite}_*}(\underline{\finite}_*,\underline{\overline{\sC}}\ttimes)\simeq \underline{\func}(\underline{\Gamma}\ttimes,\underline{\sC})
    \end{equation}
    
    Furthermore, by \cite[Prop. 9.7]{shahThesis}, we have
    \small\[\underline{\overline{\sC}}\ttimes_{[U\rightarrow G/H]}\simeq \underline{\func}_{[U\rightarrow G/H]}\big(\underline{\Gamma}\ttimes_{[U\rightarrow G/H]}, (\underline{\sC}\times\underline{\finite}_*)_{[U\rightarrow G/H]}\big)\simeq \underline{\func}(\underline{\Gamma}\ttimes_{[U\rightarrow G/H]}, \underline{\sC}_H)\]\normalsize
    If $\underline{\sC}$ has all finite indexed products, we define $\underline{\sC}\ttimes$ to be the full subcategory of $\overline{\underline{\sC}}\ttimes$ whose objects over $[U\rightarrow G/H]$ are the $H$--functors $F\colon \underline{\Gamma}\ttimes_{[U\rightarrow G/H]}\rightarrow \underline{\sC}_H$ such that for every $K\leq H$ and $K$--object $[U\rightarrowtail V]_{G/K}$ with $G$--orbit decomposition $V = \coprod_{j=1}^n V_j$ and structure maps $v_j\colon V_j\rightarrow G/K$, the map    \begin{equation}\label{eqn:conditionCartSymmMonStructure}
        F([U\rightarrowtail V]_{G/K}) \longrightarrow \prod_{j=1}^n{v_{j*}}F([U\rightarrowtail V\rightarrowtail V_j]_{V_j})
    \end{equation}
    induced by the characteristic maps $\chi_{[V_j\subseteq V]}\colon V\rightarrowtail V_j$ is an equivalence. 
    
    Now observe that, writing $U= \coprod_jU_j$ for the $G$--orbital decomposition with structure maps $u_j\colon U_j\rightarrow G/H$, we have the full subcategory $\coprod_j{u_{j!}}\underline{\ast}\subseteq \underline{\Gamma}\ttimes_{[U\rightarrow G/H]}$ consisting of the single $G$--orbits of $U$. A straightforward unwinding of definitions show that $\underline{\sC}\ttimes_{[U\rightarrow G/H]}\subseteq\underline{\overline{\sC}}\ttimes_{[U\rightarrow G/H]}$ is identified with the full subcategory 
    \[\prod_j{u_{j*}}u_j^*\underline{\sC} \simeq \underline{\func}(\coprod_j{u_{j!}}\underline{\ast}, \underline{\sC})\subseteq \underline{\func}(\underline{\Gamma}\ttimes_{[U\rightarrow G/H]},\underline{\sC}_H)\]
    where the inclusion is by right Kan extension (compare with the proof of \cite[Prop. 2.4.1.5 (4)]{lurieHA}). This in particular means that we have an identification ${\sC}\ttimes_{[U\rightarrow G/H]}\simeq \prod_j\sC_{U_j}$.
\end{construction}

\begin{remark}\label{rmk:concreteDescriptionCocartesianMorphism}
    By \cite[Recoll. 4.3]{shahPaperII} and using that the cocartesian pushforward functors to the constant $G$--cocartesian fibration $\underline{\sC}\times \underline{\finite}_*\rightarrow\underline{\finite}_*$ are just the identity functors, we see that for a morphism of $H$--objects $f\colon U \rightarrow V$ in $\underline{\finite}_{*H}$, the associated cocartesian pushforward functor on the $G$--cocartesian fibration $\underline{\overline{\sC}}\ttimes\rightarrow\underline{\finite}_*$  looks like
    \[\underline{\func}_H(\underline{\Gamma}\ttimes_{[U\rightarrow G/H]}, \underline{\sC}_H)\longrightarrow \underline{\func}_H(\underline{\Gamma}\ttimes_{{[V\rightarrow G/H]}}, \underline{\sC}_H)\quad ::\quad F \mapsto F\circ f^!\]
    where $f^! \colon \underline{\Gamma}\ttimes_{{[V\rightarrow G/H]}}\rightarrow \underline{\Gamma}\ttimes_{[U\rightarrow G/H]}$ is the $H$--functor described in \cref{cons:GammaTimesConstruction}.
\end{remark}

\begin{proposition}[``{\cite[Prop. 2.4.1.5]{lurieHA}}'']\label{prop:GCartesianSymmMonCat}
    Let $\underline{\sC}$ be a $G$--category with finite indexed products. The composite $\underline{\sC}\ttimes\subseteq \overline{\underline{\sC}}\ttimes\rightarrow\underline{\finite}_*$ is a $G$--symmetric monoidal structure on the $G$--category $\underline{\sC}$.
\end{proposition}
\begin{proof}
    By the definition of $G$--symmetric monoidal categories \cite[Def. 2.1.7 and Def. 2.2.3]{nardinShah}, first note that it would suffice to show that the composite is a $G$--cocartesian fibration and that the characteristic maps associated to any orbital decomposition $U = \coprod_{j=1}^nU_j$ induce equivalences ${\sC}\ttimes_{[U\rightarrow G/H]}\xrightarrow{\simeq}\prod_{j=1}^n\sC_{U_j}$ since these two conditions together ensure that \cite[Def. 2.1.7 (3)]{nardinShah} holds. The second point has been dealt with at the end of \cref{cons:overlineCartesianSymmetricMonoidalStructure} and so we are left to show that the composite is indeed a $G$--cocartesian fibration. 
    
    More precisely, we need to show that the condition \cref{eqn:conditionCartSymmMonStructure} is stable under the pushforward functors described in \cref{rmk:concreteDescriptionCocartesianMorphism}. To this end, suppose $F\in \underline{\func}_H(\underline{\Gamma}\ttimes_{[U\rightarrow G/H]}, \underline{\sC}_H)$ satisfies the condition \cref{eqn:conditionCartSymmMonStructure}. We need to show that for any $K\leq H$ and any $[V\rightarrowtail W]_{G/K}\in \underline{\Gamma}\ttimes_{{[V\rightarrow G/H]}}$ with structure maps $w_j\colon W_j \rightarrow G/K$, the map
    \[Ff^!([V\rightarrowtail W]_{G/K}) \longrightarrow \prod_{j=1}^n{w_{j*}}Ff^!([V\rightarrowtail W\rightarrowtail W_j]_{W_j})\]
    is an equivalence. 
    
    To set up notation, writing $f^{-1}W_j\subseteq U$ for the preimage (which might be empty) and $f^{-1}W_j= \sqcup_{i=1}^{n_j}W_{ji}$ its orbital decomposition, we know from \cref{eqn:inertActiveFactorisationPreimageFunctor} that $f^!$ is computed as the left inert map in the left square in
    \begin{center}
        \begin{tikzcd}
            U \dar[ tail]\rar["f"] & V \dar[tail] &&& W_{ji}\rar["\alpha_{ji}", rightsquigarrow]\ar[dr, "w_{ji}"'] & W_j\dar["w_j"] \\
            \sqcup_jf^{-1}W_j \rar[rightsquigarrow, "\alpha"] & W &&& & G/K
        \end{tikzcd}
    \end{center}
    together with the associated  structure maps notated on the right triangle. Using that ${w_{ji*}}\simeq {w_{j*}}{\alpha_{ji*}}$ by composability of right Kan extensions,  we now simply contemplate the following commuting diagram
    \begin{center}\footnotesize
        \begin{tikzcd}
            Ff^!([V\rightarrowtail W]_{G/K}) \rar\ar[dd, equal] & \prod_{j=1}^n{w_{j*}}Ff^!([V\rightarrowtail W_j]_{W_j}) = \prod_{j=1}^n{w_{j*}}F([U\rightarrowtail \sqcup_{i=1}^{n_j}W_{ji}]_{W_j})\dar["\simeq"]\\
            & \prod_{j=1}^n{w_{j*}}\prod_{i=1}^{n_j}{\alpha_{ji*}}F([U\rightarrowtail W_{ji}]_{W_{ji}})\dar[equal]\\
            F([U\rightarrowtail \sqcup_j\sqcup_{i=1}^{n_j}W_{ji}]_{G/K}) \rar["\simeq"] &  \prod_{j=1}^n\prod_{i=1}^{n_j}{w_{ji*}}F([U\rightarrowtail W_{ji}]_{W_{ji}})
        \end{tikzcd}
    \end{center}
    where the equivalences are by our hypothesis on $F$. Hence the top horizontal map is also an equivalence, as desired.
\end{proof}

Our next goal is to show that $\underline{\Gamma}\ttimes$ can be endowed with the structure of a $G$--cartesian pattern and to show in \cref{lem:CHII.5.14} that its monoid theory is equivalent to that associated to $\underline{\finite}_*$.

\begin{lemma}
    There is a natural factorisation system on $\underline{\Gamma}\ttimes\subseteq \underline{\finite}_*^{\Delta^1}$ where the fibrewise inert (resp. fibrewise active) morphisms are those which are pointwise fibrewise inert (resp. fibrewise active).
\end{lemma}
\begin{proof}
    The exact same argument of \cite[Lem. 5.8]{chuHaugsengII} works here since that argument only uses composability of inert morphisms and uniqueness of the fibrewise inert--active factorisations, both of which are true in $\underline{\finite}_*$.
\end{proof}

\begin{construction}[``{\cite[Rmk. 5.11]{chuHaugsengII}}'']
    We endow $\underline{\Gamma}\ttimes$ with a $G$--algebraic pattern structure using the factorisation system above with $[G/H=G/H]_{G/H}$ for all $H\leq G$ as the elementary objects. Moreover, it is not hard to see that $\eval_1\colon \underline{\Gamma}\ttimes\rightarrow \underline{\finite}_*$ endows $\underline{\Gamma}\ttimes$ with a $G$--cartesian pattern structure. To wit, for any $H$--object $[U\rightarrowtail V]_{G/H}$ in $\underline{\Gamma}\ttimes$, any inert map to an elementary object $[W=W]_W$
    \begin{center}
        \begin{tikzcd}
            U\dar[tail]\rar[tail] & W\dar[equal] \\
            V \rar[tail] & W
        \end{tikzcd}
    \end{center}
    is totally determined by the inert map $V \rightarrowtail W$, and hence the map $(\underline{\Gamma}\ttimes)\elementary_{[U\rightarrowtail V]_{G/H}/}\xrightarrow{\eval_1} \underline{\finite}\elementary_{*, V/}$  in the definition of a $G$--cartesian structure is indeed an equivalence.
\end{construction}


\begin{construction}\label{lem:uniqueLiftingActive
Morphism}
    Let $i\colon \underline{\finite}_*\hookrightarrow \underline{\Gamma}\ttimes$ be the functor that takes a finite $H$--set $U$, for any $H\leq G$, to $[U\xrightarrow{=}U]_{G/H}$. In other words, it is the right Kan extension along the inclusion $\{1\} \hookrightarrow \Delta^1$ and hence is fully faithful by \cite[Prop. 10.6]{shahThesis}. This is immediately seen to be a morphism of $G$--cartesian patterns. By the same argument as in \cite[Rmk. 5.13]{chuHaugsengII}, which uses only the uniqueness of the fibrewise inert--active factorisation in $\underline{\Gamma}\ttimes$, we see that $i$ has unique lifting of fibrewise active morphisms in the sense of \cref{defn:uniqueActiveLifting}.
\end{construction}

Now, note that a functor $M\colon \underline{\Gamma}\ttimes\rightarrow\underline{\sC}$ is a $\underline{\Gamma}\ttimes$--monoid in the sense of \cref{defn:OMonoids} if and only if for any $H$--object $[U\rightarrowtail W]_{G/H}$ of $\underline{\Gamma}\ttimes$ with orbital decomposition $W = \coprod_{j=1}^nW_j$ and structure maps $w_j\colon W_j\rightarrow G/H$, the canonical map of $H$--objects
\begin{equation}\label{eqn:monoidCanonicalMap}
    M([U\rightarrowtail W]_{G/H})\longrightarrow \prod^n_{j=1}{w_{j*}}M([W_j=W_j]_{W_j})
\end{equation}
is an equivalence. Taking this description as well as \cref{lem:activeLiftMonoidRightKanExtend} as the appropriate replacements, the proof of \cite[Lem. 5.14]{chuHaugsengII} now works in our setting \textit{mutatis mutandis} to yield:

\begin{lemma}\label{lem:CHII.5.14}
    Let $\underline{\sC}$ have finite indexed products. The adjunction $i^* \colon \monoid_{\underline{\Gamma}\ttimes}(\underline{\sC})\rightleftharpoons \monoid_{\underline{\finite}_*}(\underline{\sC}) : i_*$
    is an equivalence.
\end{lemma}

Lastly, we relate the notion of monoids explored so far with that of algebras, which we recall now. 
\begin{recollection}\label{recollection:algebraCategories}
    Let $\underline{\sC}^{\underline{\otimes}}\rightarrow\underline{\finite}_*$ be a $G$--symmetric monoidal category (cf. for instance \cite[Def. 2.2.3]{nardinShah}). Then the category of $G$--commutative algebras $\calg_G(\underline{\sC}^{\underline{\otimes}})$ is defined to be $\func_{/\underline{\finite}_*}\inerts(\underline{\finite}_*,\underline{\sC}^{\underline{\otimes}})$, that is, the category of $\underline{\finite}_*$--sections $\underline{\finite}_*\rightarrow \underline{\sC}^{\underline{\otimes}}$ which send inert morphisms to $\underline{\finite}_*$--cocartesian morphisms (cf. \cite[Def. 2.2.1]{nardinShah} for a definition). Observe now that any inert morphism in $\underline{\finite}_*$ (cf. \cite[Def. 2.1.3]{nardinShah})
    \begin{center}\scriptsize
        \begin{tikzcd}
            U\dar & Z \rar["\simeq"]\lar \dar & W\dar\\
            G/H & G/K \lar\rar[equal] & G/K
        \end{tikzcd}
    \end{center}
    can be factored as the composition of the two inert morphisms
    \begin{center}\scriptsize
        \begin{tikzcd}
            U\times_{G/H}G/K\dar & Z\lar\rar["\simeq"]\dar& W \ar[rrd, phantom, "\circ"]\dar&& U\dar & U\times_{G/H}G/K \rar[equal]\lar \dar & U\times_{G/H}G/K\dar\\
            G/K \rar[equal]& G/K \rar[equal]& G/K && G/H & G/K \lar\rar[equal] & G/K
        \end{tikzcd}
    \end{center}
    where the left one is fibrewise, i.e. it lives in the fibre over $G/K$. Since by definition any $G$--functor $A\colon \underline{\finite}_*\rightarrow \underline{\sC}^{\underline{\otimes}}$ must send the inert morphisms of the type on the right to $\underline{\finite}_*$--cocartesian morphisms, the requirement for $A$ to be in $ \calg_G(\underline{\sC}^{\underline{\otimes}})$ can equivalently be formulated as sending the fibrewise inert morphisms to $\underline{\finite}_*$--cocartesian morphisms.
\end{recollection}

The following lemma, which is an immediate modification of \cite[Lem. 5.15]{chuHaugsengII}, will be the bridge connecting the theory of monoids and that of algebras. 

\begin{lemma}\label{lem:CHII.5.15}
    Under the natural equivalence $\underline{\func}_{/\underline{\finite}_*}(\underline{\finite}_*,\underline{\overline{\sC}}\ttimes)\simeq \underline{\func}(\underline{\Gamma}\ttimes,\underline{\sC})$ from \cref{eqn:universalPropertyOverlineCartesian},  the full subcategory $\monoid_{\underline{\Gamma}\ttimes}(\underline{\sC})$ from the right hand side is identified with $\func_{/\underline{\finite}_*}\inerts(\underline{\finite}_*,\underline{\sC}\ttimes)$ from the left hand side.
\end{lemma}
\begin{proof}
    By \cref{recollection:algebraCategories}, a functor $F\colon \underline{\finite}_*\rightarrow\overline{\underline{\sC}}\ttimes$ over $\underline{\finite}_*$ lies in $\func_{/\underline{\finite}_*}\inerts(\underline{\finite}_*,\underline{\sC}\ttimes)$ if and only if $F$ factors through the full subcategory $\underline{\sC}\ttimes$ and $F$ takes fibrewise inert morphisms to cocartesian morphisms. We can translate these requirements in terms of the corresponding functor $F'\colon \underline{\Gamma}\ttimes\rightarrow\underline{\sC}$ as the following pair of conditions: for any $H\leq G$ and $H$--object $[U\rightarrowtail W]_{G/H}$ in $\underline{\Gamma}\ttimes$,
    \begin{enumerate}
        \item Writing the orbital decomposition $W = \coprod_{j=1}^nW_j$ with structure maps $w_j\colon W_j\rightarrow G/H$, the canonical map
        \[F'([U\rightarrowtail W]_{G/H})\longrightarrow \prod^n_{j=1}{w_{j*}}F'([U\rightarrowtail W\rightarrowtail W_j]_{W_j})\]
        is an equivalence.
        \item For every $H$--inert map $Y\rightarrowtail U$ in $\underline{\finite}_*$, the morphism
        \[F'([Y\rightarrowtail U \rightarrowtail W]_{G/H})\rightarrow F'([U \rightarrowtail W]_{G/H})\]
        is an equivalence. This reinterpretation of fibrewise inerts being sent to cocartesian morphisms is again by \cref{rmk:concreteDescriptionCocartesianMorphism}.
    \end{enumerate}

    On the other hand, $F'$ is a monoid if for any $H\leq G$ and $H$--object $[U\rightarrowtail W]_{G/H}$ in $\underline{\finite}_*$ and using the notations above, the map \cref{eqn:monoidCanonicalMap}  is an equivalence. To see that this is equivalent to the first pair of conditions, observe that we have the following commuting triangles
    \begin{center}\footnotesize
        \begin{tikzcd}
            F'([U\rightarrowtail W]_{G/H}) \ar[rr]\ar[dr] && \prod^n_{j=1}{w_{j*}}F'([U\rightarrowtail W_j]_{W_j})\ar[dl]\\
            & \prod^n_{j=1}{w_{j*}}F'([W_j=W_j]_{W_j})
        \end{tikzcd}
    \end{center}
    \begin{center}\footnotesize
        \begin{tikzcd}
            F'([Y\rightarrowtail U \rightarrowtail W]_{G/H}) \ar[rr]\ar[dr] && F'([U \rightarrowtail W]_{G/H})\ar[dl]\\
            & \prod^n_{j=1}{w_{j*}}F'([W_j=W_j]_{W_j})
        \end{tikzcd}
    \end{center}
    With the ingredients set up, the rest of the proof of \cite[Lem. 5.15]{chuHaugsengII} now goes through word--for--word.
\end{proof}


We may now  deduce the desired equivalence: 

\begin{proposition}\label{prop:CMon=CAlg}
    Let $\underline{\sC}$ be a $G$--category with finite indexed products. There is a canonical equivalence $\cmonoid_G(\underline{\sC})\simeq \calg_G(\underline{\sC}\ttimes)$.
\end{proposition}
\begin{proof}
    Immediate combination of \cref{lem:CHII.5.14} and \cref{lem:CHII.5.15}, using also that $\calg_G(\underline{\sC}\ttimes)=\func_{/\underline{\finite}_*}\inerts(\underline{\finite}_*,\underline{\sC}\ttimes)$ by definition.
\end{proof}

\printbibliography
\end{document}